\documentclass[12pt, a4paper,reqno]{amsart}
\numberwithin{equation}{section}
\usepackage{etex}
\usepackage{pstricks-add}
\usepackage{latexsym}
\usepackage{amssymb}
\usepackage{mathrsfs}
\usepackage{amsmath}
\usepackage{fancybox,color}
\usepackage{enumerate}
\usepackage[utf8]{inputenc}
\usepackage{eurosym}
\usepackage{color}
\usepackage{url}
\usepackage{pgf,tikz}\usepackage{mathrsfs}\usetikzlibrary{arrows}

\newcommand{\kom}[1]{}
\renewcommand{\kom}[1]{{\bf [#1]}}
\newcommand{\be}{\begin{equation}}
\newcommand{\ee}{\end{equation}}

\usepackage{hyperref}
\hypersetup{
    colorlinks=true,                          
    linkcolor=blue, 
    citecolor=red, 
    urlcolor=blue  } 
 \def\1{\raisebox{2pt}{\rm{$\chi$}}}

\newtheorem{theorem}{Theorem}[section]

\newtheorem{lemma}[theorem]{Lemma}
\newtheorem{proposition}[theorem]{Proposition}
\newtheorem{definition}[theorem]{Definition}
\newtheorem{remark}[theorem]{Remark}

\newcommand{\R}{{\mathbb R}}

\newcommand{\E}{{\mathbb E\,}}

 \newcommand{\eps}{{\varepsilon}}
 \def\1{\raisebox{2pt}{\rm{$\chi$}}}
 
\def\vint_#1{\mathchoice%
          {\mathop{\kern 0.2em\vrule width 0.6em height 0.69678ex depth -0.58065ex
                  \kern -0.8em \intop}\nolimits_{\kern -0.4em#1}}%
          {\mathop{\kern 0.1em\vrule width 0.5em height 0.69678ex depth -0.60387ex
                  \kern -0.6em \intop}\nolimits_{#1}}%
          {\mathop{\kern 0.1em\vrule width 0.5em height 0.69678ex
              depth -0.60387ex
                  \kern -0.6em \intop}\nolimits_{#1}}%
          {\mathop{\kern 0.1em\vrule width 0.5em height 0.69678ex depth -0.60387ex
                  \kern -0.6em \intop}\nolimits_{#1}}}
\def\vintslides_#1{\mathchoice%
          {\mathop{\kern 0.1em\vrule width 0.5em height 0.697ex depth -0.581ex
                  \kern -0.6em \intop}\nolimits_{\kern -0.4em#1}}%
          {\mathop{\kern 0.1em\vrule width 0.3em height 0.697ex depth -0.604ex
                  \kern -0.4em \intop}\nolimits_{#1}}%
          {\mathop{\kern 0.1em\vrule width 0.3em height 0.697ex depth -0.604ex
                  \kern -0.4em \intop}\nolimits_{#1}}%
          {\mathop{\kern 0.1em\vrule width 0.3em height 0.697ex depth -0.604ex
                  \kern -0.4em \intop}\nolimits_{#1}}}

\newcommand{\aveint}[2]{\mathchoice%
          {\mathop{\kern 0.2em\vrule width 0.6em height 0.69678ex depth -0.58065ex
                  \kern -0.8em \intop}\nolimits_{\kern -0.45em#1}^{#2}}%
          {\mathop{\kern 0.1em\vrule width 0.5em height 0.69678ex depth -0.60387ex
                  \kern -0.6em \intop}\nolimits_{#1}^{#2}}%
          {\mathop{\kern 0.1em\vrule width 0.5em height 0.69678ex depth -0.60387ex
                  \kern -0.6em \intop}\nolimits_{#1}^{#2}}%
          {\mathop{\kern 0.1em\vrule width 0.5em height 0.69678ex depth -0.60387ex
                  \kern -0.6em \intop}\nolimits_{#1}^{#2}}}

\newcommand{\dist}{\operatorname{dist}}

\usepackage{pstricks-add}

\begin{document}
\title[Gradient blow-up rates]{Gradient blow-up rates and sharp gradient estimates for diffusive Hamilton-Jacobi equations}

\date{\today}

\author[Attouchi]{Amal Attouchi}
\address{Department of Mathematics and Statistics, 
University of Jyv\"askyl\"a, 40014 Finland}
\email{amal.a.attouchi@jyu.fi; amalattouchi@gmail.com}

\author[Souplet]{Philippe Souplet}
\address{Universit\'e Sorbonne Paris Nord,
CNRS UMR 7539, Laboratoire Analyse, G\'{e}om\'{e}trie et Applications,
93430 Villetaneuse, France}
\email{souplet@math.univ-paris13.fr}

\begin{abstract}
Consider the diffusive Hamilton-Jacobi equation
$$u_t-\Delta u=|\nabla u|^p+h(x)\ \ \text{ in } \Omega\times(0,T)$$
 with Dirichlet conditions,
 which arises in stochastic control problems as well as in KPZ type models. 
We study the question of the gradient blowup rate for classical solutions with~$p>2$.

We first consider the case of time-increasing solutions. 
For such solutions, the precise rate was obtained by Guo and Hu (2008) in one space dimension,
but the higher dimensional case has remained an open question (except for radially symmetric solutions in a ball).
Here, we partially answer this question by establishing the optimal estimate
$$C_1(T-t)^{-1/(p-2)}\leq \|\nabla u(t)\|_\infty \leq C_2(T-t)^{-1/(p-2)}
\eqno(1)$$
 for time-increasing gradient blowup solutions in any convex, smooth bounded domain $\Omega$ with $2<p<3$. 
 We also cover the case of (nonradial) solutions in a ball for $p=3$.
 Moreover we obtain the almost sharp rate in general (nonconvex) domains for $2<p\le 3$.
The proofs rely on suitable auxiliary functionals, combined with the following, new Bernstein-type gradient estimate
 with sharp constant:
$$|\nabla u|\le  d_\Omega^{-1/(p-1)}\bigl(d_p+C d_\Omega^\alpha\bigr)
 \ \ \text{ in } \Omega\times(0,T),\qquad d_p=(p-1)^{-1/(p-1)},
 \eqno(2)$$
 where $d_\Omega$ is the function distance to the boundary.
 This close connection between the temporal and spatial estimates (1) and (2) seems to be a completely new observation.
  
 Next, for any $p>2$, we show that more singular rates may occur for solutions which are {\it not} time-increasing.
 Namely, for a suitable class of solutions in one space-dimension, we prove the lower estimate
 $\|u_x(t)\|_\infty \ge C(T-t)^{-2/(p-2)}$.
 \end{abstract}
 
\maketitle

\section{Introduction and main results}
\subsection{Background}
In this paper we study the initial boundary value problem for the diffusive Hamilton–Jacobi 
equation:
\begin{equation}\label{maineq}
\left\{
\begin{aligned}
u_t-\Delta u&=|\nabla u|^p+ h(x), && \quad  x\in\Omega, t>0,\\
u(x,t)&=0, &&\quad x\in \partial\Omega,t>0,\\
u(x,0)&=u_0(x), &&\quad x\in \Omega,
\end{aligned}
\right.
\end{equation}
where $p>1$. Throughout this paper, it is assumed that $\Omega \subset \R^n$ ($n\ge 1$) is a $C^{2+\rho}$-smooth 
 bounded domain for some $\rho>0$, and that 
\begin{equation}\label{hyph}
h\in C^1(\overline \Omega), \quad h\geq 0.
\end{equation}
Also, if no confusion arises, we will simply denote $\|\cdot\|_\infty$ for $\|\cdot\|_{L^\infty(\Omega)}$ and $u(t)$ for~$u(\cdot,t)$.

Problem \eqref{maineq} has a rich background. First of all, let us recall that \eqref{maineq} arises in 
stochastic control problems. Namely, consider the controled $n$-dimensional stochastic differential equation
$$dX_s =\alpha_s ds+dW_s,\  \ s>0, \quad\hbox{ with } X_0 =x\in\Omega,$$
where the stochastic process $(X_s)_{s>0}$ represents the position or state of the system,
$(W_s)_{s>0}$ is a standard Brownian motion and 
$(\alpha_s)_{s>0}$ is the control
(in other words, the controler can choose the velocity of $X$).
The spatial distribution of rewards is given by a function $u_0\in C_0(\overline\Omega)$.
More precisely, at a given time horizon $s=t>0$, the final reward is
$u_0(X_t)$ if $X$ stays in $\Omega$ until time~$t$,
and $0$ otherwise. Finally, the cost of the control at each time $s$ is assumed to be
$k_p |\alpha_s|^q$ as long as $X_s$ stays in $\Omega$, 
where $q=p/(p-1)$ is the conjugate exponent of $p$ and $k_p>0$ is a normalization constant.
The goal of the controler is then to maximize the net gain
$$G_t=\chi_{\tau>t}u_0(X_t)-k_p\displaystyle\int_0^\tau|\alpha_s|^q\, ds,$$
where $\tau$ denotes the first exit time of $X$ from $\Omega$.
It is known (see \cite{BB, BdaLio, FlSo} for details) 
that the maximal gain, also called value function of the stochastic control problem,
is given by the unique global (continuous) viscosity solution $u$ of \eqref{maineq} with $h=0$, namely:
$$u(x,t)=\sup_{(\alpha_s)_s}  \E{\hskip 0.5pt}\bigl(G_t\, | \, X_0=x\bigr),$$
where $\E{\hskip 0.5pt}\bigl(\,\cdot\,| \, X_0=x\bigr)$ denotes the conditional expectation 
with respect to the event $\{X_0=x\}$,
and the supremum is taken over all (admissible) controls.

As another motivation, \eqref{maineq} corresponds to the so-called deterministic KPZ equation,
arising in a well-known model of surface growth by ballistic deposition (see \cite{kpzhang}, \cite{KS}).
Finally, \eqref{maineq} can be seen as one the simplest model parabolic problems
with first order nonlinearity and, from the point of view of
nonlinear parabolic theory, it is thus important to understand its properties
(cp.~for instance with the extensively studied equation with zero order nonlinearity $u_t-\Delta u=u^p$).

Here our concern is about the behavior of classical solutions arising from sufficiently smooth initial data.
For any $u_0\in X$, where
$$
X:=\bigl\{u_0\in C^1(\bar\Omega);\ u_0\ge 0,\ u_0=0\hbox{ on $\partial\Omega$}\bigr\},
$$
problem \eqref{maineq} is locally well posed. 
Namely, there exists a unique maximal, classical solution $u\ge 0$.
We denote by $T=T(u_0)\in (0,\infty]$ its existence time.
When $p>2$, it is known that, for suitably large initial data, solutions may blow up 
in finite time, i.e., $T(u_0)<\infty$, in which case
$$\lim_{t\to T_-}\|\nabla u(t)\|_\infty=\infty$$
(whereas all solutions are global for $p\in (1,2]$). It is also known that the function
itself remains bounded, while its spatial gradient is the quantity to become unbounded.
The blowup phenomenon that occurs for solutions of \eqref{maineq} is usually referred to as gradient blowup (GBU). 

\subsection{Known results on gradient blow-up}
Finite time blowup phenomena for \eqref{maineq} have attracted a lot of attention in the past twenty years. Results include blowup criteria
 \cite{Alaa, ABG, souplet2002, HM04}, blowup locations \cite{Esteve, LiSouplet, SZ}, blowup profiles \cite{ABS, CG,PSpro,PS3, PSloss, SZ},
 continuation after GBU \cite{BdaLio, PZ, PSloss, QR16, PS3, FPS19}, 
 infinite time GBU \cite{SZ, SV}.  
 See also \cite{Dl, Ku, FL94, Gi, AF, AI, FTW, Att1, ZhangHu, Att2, AttBa, LYZ, AttSou1, FLa} for GBU studies for other equations.
 As a consequence of interior gradient estimates \cite{SZ}, it is known that  GBU for problem \eqref{maineq}
 can only take place on the boundary $\partial\Omega$.
For $h\equiv 0$, it is also known from \cite{BdaLio} that the solution can be extended for $t > T$
as a global weak solution (in the viscosity sense
-- its existence was already mentioned above  in connection with the stochastic control problem). Moreover,
this global weak solution becomes a classical solution again \cite{PZ} for all $t$ sufficiently large.

 The question of the gradient blowup rates for problem \eqref{maineq} as $t\to T_-$ is only partially understood.
The lower estimate
\begin{equation}\label{lowerGBU}
\|\nabla  u(t)\|_\infty\geq C (T- t)^{-1/(p-2)},\quad 0<t<T,
\end{equation}
is true for any GBU solution. 
This in particular implies that GBU is always of Type~II, i.e. it does not follow the natural self-similar scaling of the equation
(which would lead to the smaller exponent $1/2(p-1)$ instead of $1/(p-2)$).
Estimate \eqref{lowerGBU} was first established in one space dimension in \cite{CG} by a method of intersection-comparison.
In higher dimension, the weaker estimate 
$$\|\nabla  u\|_{L^\infty(\Omega\times(0,t))}\geq C (T- t)^{-1/(p-2)},\quad 0<t<T,$$
was then proved in \cite{GH} by a method based on regularity estimates (see also \cite{QSbook07} for an alternative proof).
The full lower estimate \eqref{lowerGBU} was finally obtained in \cite{PS3} by combining semigroup arguments and regularity estimates
(the result is stated there for $h=0$, but the proof immediately carries over to the general case).

Upper bounds for the GBU rate are known only in one space dimension. 
The upper bound corresponding to \eqref{lowerGBU} was first conjectured in \cite{CG} on the basis of numerical simulations
and the first analytical result in that direction was obtained in~\cite{GH}.
Namely, considering the problem with inhomogeneous boundary conditions in $\Omega=(0,1)$:
\begin{equation}\label{maineqinhom}
\left\{\begin{aligned}
u_t-u_{xx}&=|u_x|^p, && \quad  x\in(0,1), t>0,\\
u(0,t)=0,\ \ u(1,t)&=M, &&\quad t>0,\\
u(x,0)&=u_0(x), &&\quad x\in (0,1),
\end{aligned}
\right.
\end{equation} 
it was proved in \cite{GH} that any time-increasing GBU solution (i.e., $u_t\ge 0$) satisfies
\begin{equation}\label{upperGBU}
\|u_x(t)\|_\infty\le C (T- t)^{-1/(p-2)},\quad 0<t<T,
\end{equation}
and that such solutions exist for $M>0$ sufficiently large.
However no time-increasing GBU solutions can exist for $M=0$, but an analogous result
was given in \cite{QSbook07} for the original problem \eqref{maineq} in $(0,1)$ with $h$ a sufficiently large positive constant. Then,
for problem \eqref{maineq} in $(0,1)$ with $h=0$, the upper estimate \eqref{upperGBU} was obtained in \cite{PS3} for a suitable class of 
initial data (the corresponding solutions only satisfy $u_t\ge 0$ in a neighborhood of the boundary;
the proof involves the zero-number of the function $u_t$).
On the other hand, the analogue of estimate \eqref{upperGBU} was obtained in \cite{ZhangLi} 
for radially symmetric solutions of \eqref{maineq} in a ball,
which is still an essentially one-dimensional situation,
under the assumption $u_t\ge 0$ (and $h>0$ sufficiently large).
The question whether the upper estimate \eqref{upperGBU} should hold for any GBU solution of \eqref{maineq} in $\Omega=(0,1)$
was answered negatively in \cite{PS3}.
Namely, it was shown that for $h=0$, there exists a class of solutions such that
\begin{equation}\label{upperGBUinfty}
\lim_{t\to T} (T- t)^{1/(p-2)}\|u_x(t)\|_\infty=\infty.
\end{equation}
Moreover this indicates that the assumption $u_t\ge 0$ in the above results is not technical.

The upper GBU rate is a completely open problem in dimension $n\ge 2$ for nonradial solutions.
And for $n=1$, it is also unknown what are the actual rates of the more singular solutions which satisfy \eqref{upperGBUinfty}.
The main goal of this paper is to give some answers to both problems.

\subsection{GBU rate for time-increasing solutions in any space dimension}

We start with the following optimal estimate,
in the case of (smooth bounded) convex domains with $p\in (2,3)$, or symmetric domains with $p\in (2,3]$.

\begin{theorem}\label{Thmain1}
Assume \eqref{hyph} and either
\begin{equation} \label{hypOmega1}
\hbox{$p\in (2, 3)$ and  $\Omega$ is convex}
\end{equation}
or
\begin{equation} \label{hypOmega2}
\hbox{$p\in (2, 3]$ and $\Omega$ is either a ball or an annulus.}
\end{equation}
  Let $u_0\in X$ be such that $T=T(u_0)<\infty$ and 
\begin{equation} \label{hypmonot}
u_t\geq 0\quad\hbox{ in $\Omega\times (0,T)$. }
\end{equation}
Then there exist constants $C_1,C_2>0$ such that
\begin{equation}\label{GBUratecC}
C_1(T-t)^{-\frac{1}{p-2}}\leq \|\nabla u(t)\|_\infty \leq C_2(T-t)^{-\frac{1}{p-2}},\quad 0<t<T.
\end{equation}
\end{theorem}

Next, for general bounded domains with $p\in (2,3]$, our conclusion is slightly less precise,
and we have the following almost optimal result.

\begin{theorem}\label{Thmain1b}
Assume \eqref{hyph} and let $p\in (2, 3]$. Let $\Omega$ be any smooth bounded domain and let $u_0\in X$ 
be such that $T=T(u_0)<\infty$ and \eqref{hypmonot} is satisfied.
Then there exists a constant $c>0$ and, for any $\eps>0$, there exists a constant $C_\eps>0$ such that
\begin{equation}
c(T-t)^{-\frac{1}{p-2}}\leq \|\nabla u(t)\|_\infty \leq C_\eps(T-t)^{-\frac{1}{p-2}-\eps},\quad 0<t<T.
\end{equation}
\end{theorem}

\medskip

\begin{remark}\label{remGBU}
(a) It is 
an open problem whether or not Theorems \ref{Thmain1} and \ref{Thmain1b} remain true for $p>3$ and $n\ge 2$
(for nonradial solutions).
The restriction $p\le 3$ enters in the construction of our key auxiliary function (see Case~2 of the proof of Proposition~\ref{generalcompu2}).
Actually, for $p>3$, our method would allow to obtain an upper GBU rate estimate with an exponent bigger than  $1/(p-2)$. 
However, due to the gap between the upper and lower estimates in this case, and in order not to further increase the technicality of the article, we have refrained from expanding on this.

We nevertheless remark that the restriction $p\le 3$ has appeared before in some other results
on the diffusive Hamilton-Jacobi equation (see \cite[Section~IV.3]{LasryLions}, \cite{PSpro},  and cf.~also \cite{FLa})
 and that the question whether the exponent $p=3$ in those works plays a genuine critical role,
or whether such restrictions are technical, remains unclear.

\smallskip
(b) A sufficient condition for $u$ to satisfy the monotonicity assumption \eqref{hypmonot}
in Theorems \ref{Thmain1} and \ref{Thmain1b} is to take initial data $u_0\in X\cap C^2(\Omega)$ such that
\begin{equation}\label{hypmonot2}
\Delta u_0+|\nabla u_0|^p+h\ge 0 \quad\hbox{ in $\Omega$}
\end{equation}
(see for instance \cite[Section 52]{QSbook07}).
We note in particular that, for any given $u_0\in X\cap C^2(\Omega)$, \eqref{hypmonot2} along with $T(u_0)<\infty$
 is easily satisfied by taking $0\le h\in C^1(\overline\Omega)$ suitably large.
 \smallskip
 
(c) Theorems \ref{Thmain1} and \ref{Thmain1b} remain true (same proof) if we replace assumption \eqref{hypmonot} with the weaker property that
\begin{equation}\label{hypmonot3}
u_t\ge 0 \quad\hbox{ in $(T-\eta,T)\times \Omega_\eta$}
\end{equation}
 for some $\eta>0$, where $\Omega_\eta:=\left\{x\in \Omega,\ {\rm dist}(x, \partial\Omega)<\eta\right\}$.
This is of interest in the homogeneous case $h\equiv 0$, where no GBU solution of \eqref{maineq} can satisfy property \eqref{hypmonot}
 (indeed, this would imply $-\Delta u\le |\nabla u|^p$, with zero boundary conditions, hence $u\le 0$ by the maximum principle).
 However, it is a nontrivial task to verify \eqref{hypmonot3}.
By using the results in the present paper, combined with the techniques in \cite{PS3},
the existence of some classes of (radial and nonradial) GBU solutions of \eqref{maineq} in a ball for $h\equiv 0$ and $p\in (2,3]$,
satisfying \eqref{hypmonot3} and the sharp blowup rate \eqref{GBUratecC},
will be established in the forthcoming publication \cite{AttPrep}.
\end{remark}

\subsection{More singular GBU rates for solutions without time monotonicity}

Our second result concerns more singular rates for solutions without time monotonicity.
To this end let us recall the notion of {\it minimal GBU solution}:

\begin{definition}\label{defmin}
A solution $u$ of \eqref{maineq} is called a minimal GBU solution if  
 $T(u_0)<\infty$ and every initial data 
 $v_0\in X$ such that $v_0\leq u_0$ and $v_0\not\equiv u_0$ gives rise to a global classical solution, i.e. $T(v_0)=\infty$.
\end{definition}

The existence and properties of minimal GBU solutions were studied in \cite{PS3, FPS19}.
In particular the following was shown in \cite{PS3}.

\begin{proposition}\label{propmin}
Let $p>2$, $h=0$ and let $\phi\in X$, $\phi\ne 0$. 
Set 
$$\lambda^*=\sup\{\lambda\ge 0;\ T(\lambda\phi)=\infty\}.$$
Then $\lambda^*\in(0,\infty)$ and $T(\lambda^*\phi)<\infty$.
Moreover, the solution with initial data $\lambda^*\phi$ is a minimal GBU solution.
\end{proposition}

Under additional assumptions (cf.~\eqref{hypID1a}-\eqref{hypID2a} below), it was shown in \cite{PS3} that minimal GBU solutions are 
immediately regularized and then remain classical forever.
In that sense, they can be seen as an analogue of the {\em peakings solutions} for the semilinear heat equation 
\be\label{NLH}
u_t-\Delta u=u^p
\ee
(see, e.g., \cite{GV97, FMP, MatanoMerle}).
These peaking solutions  blow up only at  
one instant of time, and have a classical continuation afterwards. 
They represent a transient and minimal form of blow-up.

Now for $n=1$ and $\Omega=(0,1)$, consider the class of initial data $u_0\in W^{3,\infty}(0,1)$  
satisfying the following properties:
\be\label{hypID1a}
\hbox{$u_0$ is symmetric w.r.t. $x=\frac12$,\ \ $u_0'\ge 0$ on $[0,\frac12]$,\ \  $u_0(0)=u_0''(0)+{u_0'}^p(0)=0$,}
\ee
\be\label{hypID2a}
\hbox{there exists $a\in (0,1/2)$ such that } 
u_0''+{u_0'}^p
\begin{cases}
\,\ge 0 & \hbox{on $[0,a]$} \\
\,\le 0 & \hbox{on $[a,\frac12]$.} \\
\end{cases}
\ee
It was shown in \cite[Theorem~2.6]{PS3} that for any such $u_0$, if the corresponding solution is minimal,
then its GBU rate satisfies \eqref{upperGBUinfty}.\footnote{Assumptions \eqref{hypID1a}-\eqref{hypID2a} are motivated by intersection-comparison or zero-number arguments crucially used in the proof.}
However, no precise rate estimate was obtained.
The following result improves \cite[Theorem~2.6]{PS3}.

\begin{theorem} \label{thmfastratemain}
Let $p>2$, $h\equiv 0$ and $\Omega=(0,1)$. Let $u_0\in W^{3,\infty}(0,1)$ satisfy \eqref{hypID1a}-\eqref{hypID2a}.
Assume that $T=T(u_0)<\infty$ and that $u$ is a minimal blowup solution.
Then there exists a constant $C>0$ such that
\be \label{uxmoresingular}
u_x(0,t)\ge C(T-t)^{-2/(p-2)},\quad T/2<t<T.
\ee
\end{theorem}

\subsection{Outline of proofs}
In order to prove the upper estimate in Theorem \ref{Thmain1},
we shall construct an auxiliary functional of the form
\begin{equation}\label{DefJ0}
J(x,t):= u_t-\eps u^{p-1}d_{\Omega}^{2-p}\bigr[1+d_\Omega^\kappa\bigl],\quad \hbox{ where }d_\Omega(x):=\dist(x, \partial\Omega),
\end{equation}
with $\eps,\kappa>0$ small (the case of Theorem \ref{Thmain1b} requires a slightly different version).
We note that the last term $d_\Omega^\kappa$ in \eqref{DefJ0}, which acts as a perturbation term, cannot be avoided
(see Proposition~\ref{generalcompu2} for details).
One aims at showing that $J\ge 0$ in a neighborhood of $\partial\Omega$.
Once we know that $J$ in \eqref{DefJ0} satisfies $J\ge 0$, one obtains a differential inequality in time for the normal derivative 
at any boundary point $x_0$, by dividing by 
$d_\Omega$ and letting $x\to x_0$ in the normal direction. 
The conclusion then follows rather easily after integrating this differential inequality.

To prove  that $J\ge 0$, we need to derive a suitable parabolic inequality for $J$.
This necessitates long and delicate computations.
It is remarkable that these computations, which yield the optimal {\it time} rate $1/(p-2)$,
crucially depend on the following Bernstein-type, gradient estimate {\it in space} on the solution, with {\it sharp constant}:
\begin{equation}\label{Bernstein00}
|\nabla u|\le  d_\Omega^{-1/(p-1)}\bigl(d_p+C d_\Omega^\alpha\bigr)\ \ \text{ in } \Omega\times(0,T),\qquad d_p=(p-1)^{-1/(p-1)},
\end{equation}
for some $C,\alpha>0$ (whereas a constant larger than $d_p$ in \eqref{Bernstein00} 
would only yield a nonoptimal time rate exponent, larger than $1/(p-2)$).
The estimate \eqref{Bernstein00}, which is of independent interest and was not known before except for very special cases, 
is also established in this paper, in particular for any convex domain and $p\in(2,3)$.
This close connection between the temporal and spatial estimates \eqref{GBUratecC} 
and \eqref{Bernstein00} seems to be a completely new observation.
 
Let us observe that the functional $J$ in \eqref{DefJ0} has a more involved form than the
functional $J=u_t-\eps u^p$, used in the classical work \cite{FMcleod} (see also \cite{Sp,QSbook07}) to establish the blow-up rate
for time increasing solutions of the semilinear heat equation \eqref{NLH}.
This may be seen as a counterpart of the strong difference in the nature of blow-up between equations \eqref{maineq} and \eqref{NLH}
(boundary gradient blow-up vs.~$L^\infty$ blow-up).
We point out that the functional $J$ in \eqref{DefJ0} is also quite different from the one-dimensional functional
$$G(x,t)= u_t-\eps\Bigr[u+\bigl(1+u_x^{-\sigma}(0,t)\bigr)\bigl(1-\textstyle\frac{u_x(x,t)}{u_x(0,t)}\bigr)\Bigl],$$
used in \cite{GH} for problem \eqref{maineqinhom} on $\Omega=(0,1)$. An advantage of the functional $G$ is that it works for any $p>2$,
but there seems to be no way to use a functional of this type in higher dimensions (except for radial solutions in a ball; cf.~\cite{ZhangHu}).

As for the proof of Theorem~\ref{thmfastratemain}, we use some refinements of the arguments from \cite{PS3},
which were based on zero-number properties of $u_t$ combined with gradient estimates for linear parabolic equations with drift.
Here, a new ingredient is the observation that, under the assumptions of Theorem~\ref{thmfastratemain},
the minimal solution actually satisfies $u_t\le M(T-t)$ and not just $u_t\le M$ (see Lemma~\ref{boundforut}).

The rest of the paper is organized as follows. In section 2, we set notation and gather
a number of preliminary properties.
In section 3, we establish the required sharp, Bernstein-type, gradient estimates on the solutions.
In section 4, we construct the key auxiliary functions with the help of the above gradient estimates,
and then use these auxiliary functions to prove Theorems~\ref{Thmain1} and \ref{Thmain1b}.
Finally, section 5 is devoted to the proof of Theorem~\ref{thmfastratemain}.

\section{Notation and preliminaries}\label{sect2}
We denote by $\nu(x)$ the inward unit normal vector at any point $x\in\partial\Omega$.  
We set 
$$d_\Omega(x):=\dist(x, \partial\Omega),\quad x\in \Omega,$$
and
 $$\Omega_\delta:=\left\{x\in \Omega, d_\Omega(x) < \delta\right\},\quad \delta>0.$$
The smoothness of $\partial\Omega$ implies that $d_\Omega$ is smooth in a neighborhood of the boundary, 
that is there exists $\delta_0>0$ such that $d_\Omega\in C^2(\overline{\Omega_{\delta_0}})$. Moreover, we have 
$|\nabla d_\Omega|=1$.
 Next it is well known that the inward unit normal vector field on $\partial\Omega$ can be extended to $\Omega_{\delta_0}$ by setting
$$\nu=\nabla d_\Omega.$$
For convenience, we may actually assume that $\nu$ is the restriction to $\Omega_{\delta_0}$ 
of a $C^2$ vector field defined on the whole $\overline\Omega$,
still denoted by $\nu$ without risk of confusion.
Such a vector field exists by a standard cut-off argument (taking $\delta_0>0$ smaller if necessary),
but of course it is not assumed to satisfy $|\nu|=1$ outside of $\Omega_{\delta_0}$.
Throughout the paper, we shall then use the notation
\begin{equation}\label{defunu}
u_\nu=\nu\cdot\nabla u \quad\hbox{ in $\overline\Omega\times [0,T)$.}
\end{equation}

We next define the main constants used in this paper:
\begin{align}
\beta&=\dfrac{1}{p-1},\\
c_p&=\dfrac{\beta^\beta}{1-\beta }=\dfrac{p-1}{p-2} (p-1)^{-\frac{1}{p-1}}, \label{defcp} \\
d_p&=\beta^\beta=(1-\beta)c_p=(p-1)^{-\frac{1}{p-1}},
\end{align}
and recall some basic estimates of solutions of \eqref{maineq} (see, e.g., \cite[Propositions~2.3 and 2.4]{SZ})
that are consequences of the maximum principle: 
\begin{equation}\label{bounduMP}
\|u(\cdot,t)\|_\infty\le \|u_0\|_\infty+t\|h\|_\infty,\quad 0<t<T,
\end{equation}
and
\begin{equation}\label{boundutMP}
\|u_t(\cdot,t)\|_\infty\le M_0:=\|u_t(\cdot, T/2)\|_\infty,\quad T/2<t<T.
\end{equation}

We next recall the following fractional Gagliardo-Nirenberg type interpolation inequality:
\begin{equation}\label{estimGN1}
\|\phi \|_{W^{1,\infty}(\Omega)} \leq C_0 \|\phi \|_{W^{s,q}(\Omega)}^{\mu} \| \phi \|_{L^q(\Omega)}^{1-\mu},
\quad \phi\in W^{s,q}(\Omega)
\end{equation}
valid for any $s,q\in (1,\infty)$ such that
\begin{equation}\label{estimGN2}
\mu\in (\textstyle\frac{1}{s},1),\quad q>\frac{n}{\mu s-1}.
\end{equation}
It follows by combining the usual fractional interpolation inequality (see \cite{BM} and the references therein)
$$
\|\phi \|_{W^{\mu s,q}(\Omega)} \leq C \|\phi \|_{W^{s,q}(\Omega)}^{\mu} \| \phi \|_{L^q(\Omega)}^{1-\mu},
\quad \phi\in W^{s,q}(\Omega),$$
valid for any $\mu\in(0,1)$, with the Morrey-Sobolev inequality
$$
\|v\|_{L^\infty(\Omega)} \leq C\|v\|_{W^{k,q}(\Omega)},
\quad v\in W^{k,q}(\Omega),$$
valid for any $k>n/q$, and applied with $k=\mu s-1$.

We finally give the following useful weighted $L^q$ parabolic regularity result.

\begin{lemma} \label{lemweightedregul}
Let $0<T_1<T_2$, $q\in (1,\infty)$ and set $Q=\Omega\times(0,T_2)$, $Q'=\Omega\times(T_1,T_2)$.
Let $\gamma>0$ and assume that $z\in C^{2,1}(Q)\cap C(\overline\Omega\times(0,T_2))$ satisfies
\begin{equation}\label{HypweightedD2}
d_\Omega^\gamma (z_t-\Delta z),\ d_\Omega^{\gamma-1} \nabla z \ \hbox{ and }\ d_\Omega^{\gamma-2} z\in L^q(Q).
\end{equation}Then
\begin{equation}\label{weightedD2}
d_\Omega^\gamma D^2z \ \hbox{ and }\ d_\Omega^\gamma z_t\in L^q(Q').
\end{equation}
\end{lemma} 

The result is probably known but we give the short proof for completeness.

\begin{proof} 
Setting $Q_0:=\Omega_{\delta_0}\times (0,T_2)$, $H:=\partial_t-\Delta$ and $\phi:=d_\Omega^\gamma z \in C^{2,1}(Q_0)$.
We compute
\begin{align*}
H\phi
&=d_\Omega^\gamma Hz-2\nabla(d_\Omega^\gamma)\cdot\nabla z-z\Delta (d_\Omega^\gamma) \\
&=d_\Omega^\gamma Hz-2\gamma d_\Omega^{\gamma-1}\nabla d_\Omega\cdot\nabla z
-\gamma\bigl[d_\Omega^{\gamma-1}\Delta d_\Omega+(\gamma-1) d_\Omega^{\gamma-2}|\nabla d_\Omega|^2\bigr]z
\end{align*}
in $Q_0$. Therefore,
$$|H\phi|\le d_\Omega^\gamma|Hz|+Cd_\Omega^{\gamma-1}|\nabla z|+Cd_\Omega^{\gamma-2}|z|,$$
hence $H\phi \in L^q(Q)$ by our assumption. Since $\phi=0$ on $\partial\Omega\times (0,T)$,
it follows from standard interior-boundary $L^q$ parabolic regularity that
$\phi_t=d_\Omega^\gamma z_t\in L^q(Q_1)$ and $D^2\phi\in L^q(Q_1)$,
where $Q_1:=\Omega_{\delta_0/2}\times (T_1,T_2)$. On the other hand, writing
$$(d_\Omega^\gamma z)_{ij}=(d_\Omega^\gamma)_{ij} z+(d_\Omega^\gamma)_i z_j+(d_\Omega^\gamma)_j z_i+d_\Omega^\gamma z_{ij},$$
we get 
$$d_\Omega^\gamma |z_{ij}|\le |(d_\Omega^\gamma z)_{ij}|+Cd_\Omega^{\gamma-1}|\nabla z|+Cd_\Omega^{\gamma-2}|z|
\in L^q(Q_1).$$
Since \eqref{HypweightedD2} and interior $L^q$ parabolic regularity also guarantee that $D^2z\in L^q_{loc}(\Omega\times (0,T_2])$,
the desired property \eqref{weightedD2} follows.
\end{proof}

\section{Sharp gradient estimates}\label{sect2b}
This section is devoted to gradient estimates that will play a key role in the proofs of Theorems~\ref{Thmain1} and \ref{Thmain1b}.
It has been shown in \cite[Theorem 3.2]{SZ}, by means of a local Bernstein-type argument, 
that for any maximal classical solution $u$ to \eqref{maineq}, the following 
estimate holds
\begin{equation}\label{Bernstein0}
|\nabla u|\leq C_1(p, n)d_\Omega^{-\beta}+ C_2\quad \textrm{ in}\; \Omega\times [0,T), 
\end{equation}
where $C_2=C_2(p,\Omega, \|u_0\|_{C^1}, \|h\|_{C^1})>0$.
A  more precise control on the constant $C_1$ in \eqref{Bernstein0} 
 was obtained by \cite{LiSouplet, PSpro} and lately by \cite{FPS19}. 
By the result of \cite[Theorem 1.2]{FPS19}, whose proof relies on a Liouville-type theorem and rescaling arguments, 
we have the following estimates for any smooth bounded  domain $\Omega$.
For any $\eta>0$, there exists $C=C(\eta, u_0,h, p, \Omega)>0$ such that 
\begin{equation}
\label{FiPucciSouplet}
|\nabla u|\leq (1+\eta) d_pd_{\Omega}^{-\beta}+C\quad\text{ in}\; \Omega\times (0, T).
\end{equation}
This result was stated in \cite{FPS19}
 for $h\equiv 0$  but 
 a straightforward modification of the proof gives the same conclusion in the general case. 
We stress that the upper estimate \eqref{FiPucciSouplet} is essentially optimal: indeed, it is shown in \cite{FPS19} that,  
if $a\in\partial\Omega$ is any GBU point, then
$$\lim_{s\to 0} s^\beta |\nabla u(a+s\nu_a,T)|=d_p.$$

However, the sharp value of the constant in front of the term $d_{\Omega}^{-\beta}$ of \eqref{FiPucciSouplet} will turn out to be crucial
in order to obtain the desired exponent $1/(p-2)$ in the GBU rate \eqref{GBUratecC} (and not only $1/(p-2)+\eps$ for all $\eps>0$). 
Namely, we look for the validity of estimate \eqref{FiPucciSouplet} with $\eta=0$ (up to replacing the constant $C$
by a different lower order term).
Such an estimate was obtained in \cite{PS3} in very particular situations
(namely, under suitable symmetry assumptions on the domain $\Omega\subset \R^2$ and the initial data $u_0$,
assuming in addition that $\Omega\subset\R^2_+$ has a flat part near the origin and that $u_0$ 
is sufficiently concentrated near the origin).
It turns out that we can establish the required estimate with sharp constant for any convex domain when $p\in (2,3)$.

\begin{theorem}\label{convgradest}   
Assume \eqref{hyph}, $p\in (2,3)$ and $\Omega$ convex.
   Let $u_0\in X$ be such that $T=T(u_0)<\infty$.
  Then,  for any $\alpha\in \bigl(0,\frac{3-p}{4(p-1)}\bigr)$, there exists $C>0$  
  such that
\begin{equation}\label{estconvgradest}
|\nabla u|\le  d_\Omega^{-\beta}\bigl(d_p+C d_\Omega^\alpha\bigr)
\quad\text{in}\ \Omega\times (0, T).
\end{equation}
\end{theorem}

 Next, in the special situation when the domain is invariant under rotations, 
we have a sharp estimate similar to that in Theorem~\ref{convgradest}.
Its proof will be somewhat easier than that of Theorem~\ref{convgradest} and,
unlike the latter, it is valid for any $p>2$.

\begin{theorem}\label{shgradest}
Assume that $p>2$ and that $\Omega$ is a ball or an annular domain.
   Let $u_0\in X$ be such that $T=T(u_0)<\infty$.
  Then,  for any $\alpha\in \bigl(0,\frac{\beta}{2}\bigr)$, there exists $C>0$ such that
  \eqref{estconvgradest} holds.
  \end{theorem}

To prove Theorem~\ref{convgradest}, we need two lemmas.
The first one provides a control on the tangential derivatives.
It can be proved by adapting a device from \cite{LasryLions} (for the corresponding elliptic problem),
 using the scaling of the equation, the convexity of $\Omega$ and a comparison principle.

\begin{lemma} \label{lemutau}
 Define the tangential part of the gradient by
\begin{equation}\label{deftanggrad}
\nabla_\tau u=\nabla u-(\nu\cdot\nabla u)\nu.
\end{equation}
Under the assumptions of Theorem~\ref{convgradest}, there exists $C>0$ such that
\label{tangbound1}
\begin{equation}\label{tanestconv}
 |\nabla_\tau u|\le Cd_\Omega^{-1/2}\quad\hbox{ in $\Omega_{\delta_0}\times (0,T)$.}
\end{equation}
\end{lemma}

We note that the proof of Lemma~\ref{lemutau} actually works for all $p>2$ but,  in view of \eqref{FiPucciSouplet}, estimate \eqref{tanestconv} is only of interest for $p<3$ 
(since otherwise $1/2\ge \beta$).

\begin{proof} 
We adapt the argument of \cite[Proposition IV.2]{LasryLions}.
Fix $y\in \Omega$. Let $\lambda\in (0,1)$ and set $\sigma:=\dfrac{p-2}{p-1}$.  
For any $(x,t)\in\overline\Omega\times [0,T)$, we have $(y+\lambda (x-y), \lambda^2 t)\in\overline\Omega\times [0,T)$
owing to the convexity of $\Omega$.
We may thus set
$$u_\lambda(x,t)=\lambda^{-\sigma}u(y+\lambda (x-y), \lambda^2 t),\qquad (x,t)\in\overline\Omega\times [0,T).$$
We have that in $\Omega\times [0,T)$, 
$u_\lambda$ satisfies
$$\partial_t u_\lambda -\Delta u_\lambda=|\nabla u_\lambda|^p+\lambda^{2-\sigma}h(y+\lambda(x-y)).$$
Moreover, using  
$h\in C^1(\bar \Omega)$ and the fact that $|y-x|\le {\rm diam}(\Omega)$, we get
\begin{align*}
\lambda ^{2-\sigma}h
&(y+\lambda (x-y))\\
&=h(x)-h(x)+h(x+(1-\lambda)(y-x))+(\lambda^{2-\sigma}-1)h(y+\lambda (x-y))\\
& \ge h(x)-(1-\lambda)|y-x|\|\nabla h\|_\infty-(1-\lambda^{2-\sigma}) \|h\|_\infty \ge h(x)- C(1-\lambda).
\end{align*}
Consequently
$$\partial_t u_\lambda -\Delta u_\lambda-|\nabla u_\lambda|^p\ge h(x)-C(1-\lambda)\quad\hbox{ in $\Omega\times (0,T)$.}$$
Next, using $u_0\ge 0$, $u_0\in C^1(\bar \Omega)$, and the fact that
 $\lambda^{-\sigma}\ge 1$ we obtain that, for all $x\in\Omega$,
\begin{align*}
u_\lambda(x,0)=\lambda^{-\sigma}u_0(y+\lambda(x-y))
&\ge u_0(x+(1-\lambda)(y-x))\\
&=u_0(x)-u_0(x)+u_0(x+(1-\lambda)(y-x))\\
&\ge u_0(x)-C(1-\lambda) \quad\hbox{ for all $x\in\Omega$}.
\end{align*}
Since, moreover,
$$u_\lambda(x,t)\ge 0=u(x,t)\quad\hbox{ for all $(x,t)\in \partial\Omega\times (0,T)$},$$
it follows that $$W(x, t):=u_\lambda(x,t)+C(1-\lambda)(t+1)$$
is a supersolution to \eqref{maineq}.
From the comparison principle we get that $W\ge u$, that is:
$$u(x+(1-\lambda)(y-x), \lambda^2t)\ge \lambda^\sigma u(x,t)-C\lambda^\sigma(1-\lambda)(t+1).$$
Using the bounds \eqref{bounduMP}, \eqref{boundutMP}, it follows that
\begin{align*}
u\bigl(x+&(1-\lambda)(y-x),t\bigr)-u(x,t)\\
&\ge u\bigl(x+(1-\lambda)(y-x), \lambda^2t\bigr)-u(x,t)-C_1(1-\lambda^2)t\\
& \ge (\lambda^\sigma-1)u(x,t)-C\lambda^\sigma(1-\lambda)(t+1)-C_1(1-\lambda^2)t\\
& \ge -C_2(1-\lambda).
 \end{align*}
Dividing by $(1-\lambda)$ and  sending $\lambda\to 1$,
we get that for any $x,y\in \Omega$, it holds
\begin{equation}\label{boundeqconv}
(y-x)\cdot\nabla u(x,t)  \ge -C_2.
\end{equation}
 Now, since $\Omega$ is smooth, it satisfies an interior sphere condition of radius $R$ for some $R>0$.
For any $x\in \Omega_R$  and any vector $\xi \perp \nu(x)$,
we see that the segment $(x-s\xi,x+s\xi)\subset \Omega$ with 
$$s=\sqrt{R^2-(R-d_\Omega(x))^2}=\sqrt{2Rd_\Omega(x)-d^2_\Omega(x)}>\sqrt{Rd_\Omega(x)}=:s_0(x).$$
 For any $x\in \Omega_R$ and $t\in (0,T)$, take $\xi:=\frac{\nabla_\tau u(x,t)}{|\nabla_\tau u(x,t)|}\perp \nu(x)$ (in case $\nabla_\tau u(x,t)\ne 0$,
otherwise there is nothing to prove). Choosing $y=x-s_0\xi$ in \eqref{boundeqconv}, we deduce 
(omitting the variables $x,t$ for conciseness) that
$$\sqrt{Rd_\Omega}|\nabla_\tau u|
=\sqrt{Rd_\Omega}\frac{\nabla_\tau u}{|\nabla_\tau u|}\cdot\bigl(\nabla_\tau u+(\nu\cdot\nabla u)\nu\bigr)
=s_0\xi\cdot\nabla u\le C_2,$$
hence \eqref{tanestconv}
(using also \eqref{Bernstein0} in case $R<\delta_0$).
\end{proof} 

Next, we want to estimate the normal derivative. 
The main idea is to use the PDE together with an estimate for the tangential part of $\Delta u$
to derive  a differential inequality for $u_\nu$.
Our second lemma provides the required estimate on the tangential part of $\Delta u$.

\begin{lemma} \label{boundtautau}
Assume that 
\begin{equation}\label{tanestconv2} 
|\nabla_\tau u|\le Cd_\Omega^{-k}\quad\hbox{ in $\Omega_{\delta_0}\times (0,T)$,}
\end{equation}
for some $k\in [0,\beta)$.
Then for any $\gamma>\frac{\beta+k}{2}$,
there exists a constant $C_1>0$ such that 
\begin{equation}\label{utautau}
 |\Delta u-u_{\nu\nu}|\le C_1 d_\Omega^{-1-\gamma}\quad\hbox{ in $\Omega_{\delta_0}\times (T/2,T)$.}
\end{equation}
where $u_{\nu\nu}=\nu(D^2u)\nu$.
In particular, under the assumptions of Theorem~\ref{convgradest}, inequality~\eqref{utautau} is true for any $\gamma>\gamma_0:=\frac{\beta}{2}+\frac14$.
\end{lemma}

 We note that $\gamma_0<\beta$ for $p\in(2,3)$.
To prove Lemma~\ref{boundtautau}, the rough idea is to first estimate the higher derivatives $D^3u$ by combining
parabolic regularity and the Bernstein estimate~\eqref{Bernstein0},
and then to interpolate $D(\nabla_\tau u)$ between $D^3u$ and $\nabla_\tau u$. 
However, for technical reasons, we cannot conveniently estimate $D^3u$ itself (unless requiring higher regularity on $h$),
but we shall ``almost'' estimate it by working in the Sobolev spaces $W^{3-\eta,q}$ 
with $\eta>0$ small.

\begin{proof}  Let us define 
 $$\Omega^\eps:=\left\{x\in \Omega,\ d_\Omega(x)>\eps\right\},\quad \eps>0.$$
For convenience, we shall as usual denote partial derivatives of functions   by subscripts, that is   $u_i=\frac{\partial u}{\partial x_i}=\partial_i u$.  Also the spatial gradient will be denoted indifferently by $D$ or $\nabla$.

{\bf Step 1.} {\it Higher Sobolev estimate of $u$.} We claim that, for any $\eta\in (0,1)$ and $q\in (1,\infty)$, there exists a constant $C>0$ such that
\begin{equation}\label{estimD3u}
\|u(\cdot,t)\|_{W^{3-\eta,q}(\Omega^\eps)}\le C\eps^{-\beta-2},\quad T/2<t<T,\ \eps\in(0,\delta_0).
\end{equation}

By \eqref{Bernstein0}, the functions
$d_\Omega^{\beta-1} |u|$, $d_\Omega^\beta |\nabla u|$ and $d_\Omega^{\beta+1}|u_t-\Delta u|=d_\Omega^{p\beta} |\nabla u|^p$
are bounded in $\Omega\times (0,T)$. We may thus apply Lemma~\ref{lemweightedregul} with $z=u$, $\gamma=\beta+1$ to deduce that
for any $q\in (1,\infty)$,
\begin{equation}\label{estimD3uA}
d_\Omega^{\beta+1} D^2u
\in L^q(Q'),\quad\hbox{where $Q':=\Omega\times (T/4,T)$}.
\end{equation}
Therefore, for any $i\in \{1,\dots,n\}$, we have
$$d_\Omega^{\beta+2}|\partial_tu_i-\Delta u_i|\le d_\Omega^{\beta+2}
\bigl[p|\nabla u|^{p-1}|\nabla u_i|+|h_i|\bigr]\le Cd_\Omega^{\beta+1}(|D^2u|+1)
\in L^q(Q').$$
Since also $d_\Omega^\beta |u_i|$ and $d_\Omega^{\beta+1} |\nabla u_i|\in L^q(Q')$,
we may thus apply Lemma~\ref{lemweightedregul} with $z=u_i$ and $\gamma=\beta+2$ to deduce that
\begin{equation}\label{estimD3uC}
d_\Omega^{\beta+2} D^3u \ \hbox{ and }\ d_\Omega^{\beta+2} \partial_tDu\in L^q(\Omega\times (T/2,T)).
\end{equation}
Now, setting $I=(T/2,T)$, it follows from \eqref{Bernstein0}, \eqref{estimD3uA}, \eqref{estimD3uC} that
\begin{equation}\label{estimD3uE}
\|Du\|_{L^q(I;W^{2,q}(\Omega^\eps))}+\|Du\|_{W^{1,q}(I;L^q(\Omega^\eps))}\le C\eps^{-\beta-2},\quad \eps\in (0,\delta_0)
\end{equation}
(here and hereafter, the generic constant $C$ is independent of $\eps$ but may depend on $q,\eta$).
For any $\theta\in (0,1)$, we next use the imbedding 
$$L^q(I;W^{2,q}(\Omega^\eps))\cap W^{1,q}(I;L^q(\Omega^\eps)) \hookrightarrow W^{\theta,q}(I;W^{2(1-\theta),q}(\Omega^\eps)).$$
Observe that the constant in this imbedding can be chosen uniform for all $\eps\in(0,\eps_0)$, 
with $\eps_0\in (0,\delta_0)$ sufficiently small
(this follows from the uniform boundedness of the corresponding second order extension operators).
It thus follows from \eqref{estimD3uE} that
\begin{equation}\label{estimD3uF}
\|Du\|_{W^{\theta,q}(I;W^{2-2\theta,q}(\Omega^\eps))}\le C\eps^{-\beta-2},\quad \eps\in(0,\delta_0).
\end{equation}
Since for $q>1/\theta$ we have the imbedding
$W^{\theta,q}(I;W^{1-2\theta,q}(\Omega^\eps))\subset L^\infty(I;W^{1-2\theta,q}(\Omega^\eps))$
(uniform in $\eps$),
the claim \eqref{estimD3u} follows.

{\bf Step 2.} {\it Estimate of $\Delta u-u_{\nu\nu}$.}
We shall denote by $\nu^i$ the coordinates of $\nu$ and use the convention of summation on repeated indices.

Recalling \eqref{defunu} and \eqref{deftanggrad}, we can define the functions $\nabla_\tau u$ and $u_{\nu\nu}$ 
in $\Omega\times (0,T)$ and we have 
$\nabla_\tau u=\nabla u-(\nu\cdot\nabla u)\nu=\bigl(u_i-u_j\nu^i\nu^j\bigr)e_i$.
It follows that 
$$\Delta u-u_{\nu\nu}\equiv u_{ii}-u_{ij}\nu^i\nu^j
=\bigl(u_i-u_j\nu^i\nu^j\bigr)_i+u_j(\nu^i\nu^j)_i
=\nabla\cdot(\nabla_\tau u)+u_j(\nu^i\nu^j)_i,$$
hence
\begin{equation}\label{D23u0}
|\Delta u-u_{\nu\nu}|\le C\bigl(|D(\nabla_\tau u)|+|\nabla u|\bigr).
\end{equation}

Let $q\in(1,\infty)$, $\eta\in(0,1)$, $t\in (T/2,T)$ and $\eps\in (0,\delta_0)$. 
In what follows $C>0$ will denote a generic constant independent of $t,\eps$ but depending on $q,\eta$ (and $\mu$ below).
We shall estimate $\nabla_\tau u(\cdot,t)$ in $W^{1,\infty}(\Omega^\eps)$ 
by interpolating between its $L^\infty$ norm
and its $W^{2-\eta,q}$ norm for $\eta$ small and $q$ large.
To this end, for each $i,k\in\{1,\dots,n\}$, we first compute
$$\partial_k\bigl(u_j\nu^i\nu^j-u_i\bigr)
=u_{jk}\nu^i\nu^j+u_j(\nu^i\nu^j)_k-u_{ik},$$
so that, by \eqref{estimD3u},
\begin{equation}\label{D23u}
\|\nabla_\tau u(\cdot,t)\|_{W^{2-\eta,q}(\Omega^\eps)}\le C\|u(\cdot,t)\|_{W^{3-\eta,q}(\Omega^\eps)}
\le C\eps^{-\beta-2}.
\end{equation}
We now use the Gagliardo-Nirenberg type interpolation inequality \eqref{estimGN1}-\eqref{estimGN2}
with $k=2-\eta$ which implies
$$
\|\phi \|_{W^{1,\infty}(\Omega^\eps)} \leq C_0 \|\phi \|_{W^{2-\eta,q}(\Omega^\eps)}^{\mu} \| \phi \|_{L^\infty(\Omega^\eps)}^{1-\mu},
\quad \phi\in C^2(\overline\Omega^\eps),
$$
valid for any
\begin{equation}\label{condGN}
\mu\in \bigl(\textstyle\frac{1}{2-\eta},1\bigr),\quad q\in\bigl(\frac{n}{\mu(2-\eta)-1},\infty\bigr)
\end{equation}
(note that, as above, the constant $C_0=C_0(\eta,q,\mu)$ can be taken independent of $\eps>0$ small).
Applying this with $\phi=\nabla_\tau u(\cdot,t)$ and using \eqref{tanestconv2}, 
\eqref{D23u}, we get 
$$
\|D(\nabla_\tau u)(\cdot,t)\|_{L^\infty(\Omega^\eps)}
\le C_0 \|\nabla_\tau u(\cdot,t)\|_{W^{2-\eta,q}(\Omega^\eps)}^{\mu}\|\nabla_\tau u(\cdot,t)\|_{L^\infty(\Omega^\eps)}^{1-\mu} 
\le C\eps^{-\mu(\beta+2)-(1-\mu)k}. 
$$
Going back to \eqref{D23u0} and using \eqref{Bernstein0} again, we obtain
$$\|\bigl(\Delta u-u_{\nu\nu}\bigr)(\cdot,t)\|_{L^\infty(\Omega^\eps)}
\le C\eps^{-\mu(\beta+2)-(1-\mu)k}$$ 
 Since $\mu\to (\frac12)_+$ as $\eta\to 0$ in \eqref{condGN}, the lemma follows.
\end{proof}

We can now complete the proof of Theorem \ref{convgradest} in the convex domain case.

\begin{proof}[Proof of Theorem \ref{convgradest}]
 In $\Omega_{\delta_0}\times (T/2,T)$, we have
$$u_{\nu\nu}+|u_\nu|^p\leq u_{\nu\nu}-\Delta u+\Delta u+|\nabla u|^p=u_{\nu\nu}-\Delta u+u_t-h.$$ 
Pick any $\gamma\in \bigl(\frac{\beta}{2}+\frac14,\beta\bigr)$ (which is nonempty for $p\in (2,3)$).
By \eqref{utautau} in Lemma~\ref{boundtautau} and \eqref{boundutMP} it follows that
\begin{equation}\label{ineqg0}
u_{\nu\nu}+ |u_\nu|^p\leq C+C d_\Omega^{-1-\gamma}(x)\leq Md_\Omega^{-1-\gamma}(x).
\end{equation}
Now fix $(x,t)\in\Omega_{\delta_0}\times (T/2,T)$ and let $\bar x\in\partial\Omega$ be the projection of $x$ onto the boundary.
Set 
$$g(s)=\nu(\bar x)\cdot\nabla u(\bar x+s\nu(\bar x),t)
=u_\nu(\bar x+s\nu(\bar x),t), \qquad 0<s\le d_\Omega(x).$$
 Then \eqref{ineqg0} yields
\begin{equation}\label{ineqg}
g'+g^p\leq M s^{-1-\gamma},\qquad 0<s\le d_\Omega(x).
\end{equation}

Next, define the function
$$w(s):=d_ps^{-\beta}+ Ms^{-\gamma}, \qquad s>0.$$
Using $\gamma\leq \beta$, $p\beta=1+\beta$, $\beta d_p=d_p^p$, we get
$$
w'+w^p
\geq -\beta d_ps^{-1-\beta}-M\gamma s^{-1-\gamma}+d_p^ps^{-p\beta}+ Mpd_p^{p-1}s^{-(p-1)\beta-\gamma} 
=M(\beta+1-\gamma)s^{-1-\gamma},
$$
hence
\begin{equation}\label{ineqg1}
w'+w^p\ge M s^{-1-\gamma},\qquad  s>0.
\end{equation}

We now claim that
\begin{equation}\label{ineqg2}
g\leq w \quad\text{in } (0, d_\Omega(x)).
\end{equation}
  Indeed, set $f:=g-w$ and assume for contradiction that there exists $b\in(0, d_\Omega(x))$ such that $f(b)>0$.  
Since the function $g$ is bounded on $(0, d_\Omega(x))$ owing to $u(\cdot,t)\in C^1(\overline\Omega)$, there exists $\eps\in (0,b)$ 
such that $f(\eps)< 0$.
By continuity, there exists $c\in (\eps, b)$ such that $f(c)=0$ and $f>0$ on $(c,b]$.
By \eqref{ineqg} and \eqref{ineqg1}, we then get 
$f'\leq w^p-g^p<0$ on $(c,b]$, hence
$f(b)<f(c)=0$, which is a contradiction.

Through \eqref{ineqg2}, we have thus proved that
\begin{equation}\label{ineqg3}
u_\nu(x, t)\leq d_p d_\Omega^{-\beta}(x)+ Md_\Omega^{-\gamma}(x)
\quad\hbox{ in  $\Omega_{\delta_0}\times (T/2,T)$.}
\end{equation}
On the other hand, it follows from the proof of \cite[Theorem 1.3]{FPS19} that
\begin{equation}\label{ineqg4}
C_1:=\inf_{\Omega\times (0,T)}u_\nu > -\infty.
\end{equation}
By combining  \eqref{ineqg3}, \eqref{ineqg4} and \eqref{ineqg}, we obtain
$$|\nabla u(x,t)|\leq d_p d_\Omega^{-\beta}+ Md_\Omega^{-\gamma}+C_1+Cd_\Omega^{-1/2}
\le d_p d_\Omega^{-\beta}+ C_2d_\Omega^{-\gamma}
\quad\hbox{ in  $\Omega_{\delta_0}\times (T/2,T)$,}
$$
hence \eqref{estconvgradest}.
\end{proof}

 We now turn to the proof of Theorem \ref{shgradest}, 
which is easier than that of Theorem \ref{convgradest}, 
taking advantage of the invariance of the homogeneous part of the equation
under translations (see also Remark~\ref{remradial}).

\begin{proof}[Proof of Theorem \ref{shgradest}]
We proceed in two steps. First we show that the tangential derivatives are bounded. 
Then we give a precise control on the gradient in the normal direction $u_\nu$.

Let $\Theta$ be a rotation and define $u_\Theta(x,t)=u(\Theta x,t)$.
The function $u_\Theta$ solves
$$\partial_t u_\Theta-\Delta u_\Theta=|\nabla u_\Theta|^p+h\circ\Theta.$$
 Letting 
$$v=u_\Theta-\|u_0\circ\Theta-u_0\|_\infty-t\|h\circ\Theta-h\|_\infty,$$ 
we thus have 
$$v_t-\Delta v-|\nabla v|^p=h\circ\Theta-\|h\circ\Theta-h\|_\infty\le h
\quad\hbox{ in $\Omega\times (0,T)$,}$$ 
as well as $v(\cdot,0)=u_0\circ\Theta-\|u_0\circ\Theta-u_0\|_\infty\le u_0$ in $\Omega$.
Since $v\le 0$ on $\partial\Omega\times (0,T)$, it then follows from the comparison principle that
$v\le u$ hence,
$$u-u_\Theta\ge -\|u_0\circ\Theta-u_0\|_\infty-T\|h\circ\Theta-h\|_\infty
\ge -C\|\Theta-I\|_{L^\infty(\Omega)}\quad\hbox{ in $\Omega\times (0,T)$,}$$ 
where $C=\|\nabla u_0\|_\infty+T\|\nabla h\|_\infty$ and $I$ is the identity.
This easily yields
\begin{equation}\label{boundutau}
|u_\tau|\le C\quad\hbox{ in $\Omega\times (0,T)$,}
\end{equation}
for any tangential derivative.

 By Lemma \ref{boundtautau} with $k=0$, we deduce that, for any $\gamma>\frac{\beta}{2}$,
\begin{equation}\label{boundutautau2}
 |\Delta u-u_{\nu\nu}|\le Cd_\Omega^{-1-\gamma}\quad\hbox{ in $\Omega_{\delta_0}\times (T/2,T)$.}
\end{equation}
Arguing as in the proof of Theorem \ref{convgradest}, we then obtain that
$|u_\nu|\leq d_p d_\Omega^{-\beta}+ Cd_\Omega^{-\gamma}$,
hence 
$$|\nabla u|\leq |u_\nu|+|u_\tau|\leq d_p d_\Omega^{-\beta}+ Cd_\Omega^{-\gamma}$$
for any $\gamma>\frac{\beta}{2}$, and the conclusion follows.
\end{proof}
 
\begin{remark}\label{remradial}
Under the stronger assumption $h\in C^2(\overline\Omega)$, 
in the proof of Theorem~\ref{shgradest}, the boundedness of all tangential second order derivatives
$u_{\tau\tau}$ can be shown by a similar comparison argument as that leading to \eqref{boundutau}, hence giving \eqref{boundutautau2}
without making use of Lemma \ref{boundtautau}.
However, this simplification does not seem possible if $h$ is merely $C^1$.
\end{remark}

\section{Blow-up rates for time-increasing solutions: \\ proof of Theorems \ref{Thmain1} and \ref{Thmain1b}} \label{sect3}
In this section we prove Theorems \ref{Thmain1} and \ref{Thmain1b}.
By the same token, we shall actually also prove the following result.

\begin{proposition}\label{Thmain1gen}
Let $p\in (2, 3]$ and  $\Omega$ be a smooth bounded domain.
Let $u_0\in X$ be such that $T(u_0)<\infty$. Assume that there exist $\eta,\alpha,C>0$ such that
\begin{equation} \label{hypmonot3b}
u_t\geq 0 \quad\hbox{ in $\Omega_\eta\times (T-\eta,T)$}
\end{equation}
and
\begin{equation}\label{hypsharpprofile}
|\nabla u|\le d_\Omega^{-\beta}\bigl(d_p +Cd_\Omega^\alpha\bigr)\quad\hbox{ in $\Omega\times (0, T).$}
\end{equation}
Then there exist constants $C_1,C_2>0$ such that
\begin{equation}
C_1(T-t)^{-1/(p-2)}\leq \|\nabla u(t)\|_\infty \leq C_2(T-t)^{-1/(p-2)},\quad 0<t<T.
\end{equation}
\end{proposition}

Proposition~\ref{Thmain1gen} shows that the time-monotonicity is needed only in a neighborhood of the boundary (cf. Remark~\ref{remGBU}(c)).
It also shows that for $p\le 3$ the sharp gradient estimate \eqref{hypsharpprofile} automatically guarantees the sharp GBU rate.
Note that Theorem~\ref{Thmain1gen} may be of future interest, since \eqref{hypsharpprofile} might be true for general (nonconvex) domains 
although this is presently unknown.

\subsection{Construction of auxiliary functions}
Let $u$ be the maximal classical solution of \eqref{maineq}. 
We define the linear parabolic operator $\mathcal P$ as
$$\mathcal P\phi:=\phi_t-\Delta\phi-p|\nabla u|^{p-2}\nabla u\cdot\nabla \phi.$$
Notice that $\mathcal{P}(u_t)=0$ and
\begin{equation}\label{eqPu} 
\mathcal Pu=(1-p)|\nabla u|^p+h(x).
\end{equation}
We shall use the notation 
$$G=G(x,t)=|\nabla u|,\qquad X=X(x,t)=\dfrac{|\nabla u|d_{\Omega}}{u}.$$
The following lemma is the building block for constructing the various auxiliary functions that we will use.

\begin{lemma}\label{generalcompu}  
(i) For any $a,b\in\R$, we have, in $\Omega_{\delta_0}\times(0,T)$,
\begin{equation}\label{eqPu2}
\begin{aligned}
\dfrac{-\mathcal P(u^ad_\Omega^b)} {u^a d_\Omega^{b-2}}
&=a\Bigl[(p-1){d_\Omega^2G^p\over u}+(a-1)\Bigl({d_\Omega G\over u}\Bigr)^2-\dfrac{h(x)d_\Omega^2}{u}
+2b\dfrac{d_\Omega}{u}\nabla u\cdot\nabla d_\Omega\Bigr]\\
&\quad+b\Bigl[d_\Omega\Delta d_\Omega+pd_\Omega|\nabla u|^{p-2}\nabla u\cdot\nabla  d_\Omega+ (b-1)
\Bigr].\end{aligned}
 \end{equation}

(ii) Let $\sigma\in [0,p-1)$, $\kappa\in [0,p-2)$. There exists a constant $M>0$ (depending on~$u_0$), such that
$$\dfrac{-\mathcal P\bigl[u^{p-1-\sigma}d_\Omega^{2+\kappa-p}\bigr]} {u^{p-1-\sigma}d_\Omega^{\kappa-p}}
\ge F(X)+d_\Omega G^{p-1}\bigl[(p-1-\sigma)(p-1)X-p(p-2-\kappa)\bigr]-Md_\Omega$$
in $\Omega_{\delta_0}\times[T/2,T)$, where
$$F(X):=(p-1-\sigma)(p-2-\sigma)X^2-2(p-1-\sigma)(p-2-\kappa)X+(p-1-\kappa)(p-2-\kappa).$$
\end{lemma}

\begin{proof}
(i)  Recall that $d_\Omega\in C^2(\Omega_{\delta_0})$.
We compute, using \eqref{eqPu},
$$\mathcal P(u^a)=au^{a-1}\mathcal Pu-a(a-1)u^{a-2}|\nabla u|^2=-au^{a-1}\bigl[(p-1)G^p+(a-1)u^{-1}G^2-h(x)\bigr]$$
and
$$\mathcal P(d_\Omega^{b})=-b d_\Omega^{b-1}\bigl[\Delta d_\Omega+p|\nabla u|^{p-2}\nabla u\cdot \nabla d_\Omega\bigr]-b(b-1)|\nabla d_\Omega|^2d_\Omega^{b-2}.$$
Using the identity
$$\mathcal P(\phi\psi)=\psi\mathcal P\phi+\phi\mathcal P\psi-2\nabla\phi\cdot\nabla\psi$$
for any $C^{2,1}$ functions $\phi, \psi$, it follows that 
\begin{align*}
\mathcal P(u^ad_\Omega^{b})
&=-2abu^{a-1}d_\Omega^{b-1}\nabla u\cdot \nabla d_\Omega
-au^{a-1}d_\Omega^{b}\bigl[(p-1)G^p+(a-1)u^{-1}G^2-h(x)\bigr] \\
&\qquad -b u^ad_\Omega^{b-1}\bigl[\Delta d_\Omega+p|\nabla u|^{p-2}\nabla u\cdot \nabla d_\Omega\bigr]-b(b-1)|\nabla d_\Omega|^2d_\Omega^{b-2},
\end{align*}
hence \eqref{eqPu2},
 using
\begin{equation}\label{nabladOmega}
|\nabla d_{\Omega}|=1\quad\hbox{  in $\Omega_{\delta_0}$. }
\end{equation}

(ii) We have
$$|\Delta d_{\Omega}|\leq C(\Omega)\quad\hbox{  in $\Omega_{\delta_0}$.}$$
Also, by Hopf's Lemma, there exists $c_0>0$ such that
\begin{equation}\label{Hopfucd}
u\ge c_0d_\Omega\quad\hbox{ in $\Omega\times [T/2, T)$,}
\end{equation}
hence $|h(x)|d_\Omega^2u^{-1}\le c_0^{-1}\|h\|_\infty d_\Omega$.
 Recalling \eqref{nabladOmega}, it then follows from \eqref{eqPu2} with $a=p-1-\sigma>0$ and $b=2+\kappa-p<0$ that
\begin{align*}
&\dfrac{-\mathcal P\bigl[u^{p-1-\sigma}d_\Omega^{2+\kappa-p}\bigr]} {u^{p-1-\sigma}d_\Omega^{\kappa-p}}\\
&\quad\ge(p-1-\sigma)\Bigl[(p-1){d_\Omega^2G^p\over u}+(p-2-\sigma)\Bigl({d_\Omega G\over u}\Bigr)^2
-2(p-2-\kappa)\Bigl({d_\Omega G\over u}\Bigr)\Bigr]\\
&\qquad -p(p-2-\kappa)G^{p-1}d_\Omega+(p-1-\kappa)(p-2-\kappa)-Md_\Omega\\
&\quad\ge (p-1-\sigma)(p-2-\sigma)X^2-2(p-1-\sigma)(p-2-\kappa)X+(p-1-\kappa)(p-2-\kappa)\\
&\qquad +d_\Omega G^{p-1}\bigl[(p-1-\sigma)(p-1)X-p(p-2-\kappa)\bigr]-Md_\Omega
\end{align*}
in $\Omega_{\delta_0}\times [T/2, T)$, which proves assertion (ii).
\end{proof}

By suitably combining Lemma~\ref{generalcompu} with the gradient estimates of Section~2,
we shall prove the following key proposition, whose proof is rather technical
and which will enable us 
to construct our auxiliary functions.
Let us point out that the quantities on the LHS of \eqref{auxfctB1} and \eqref{auxfctB3} will provide the main terms in the auxiliary functions,
whereas that in \eqref{auxfctB2}, which has a different homogeneity and provides a negative contribution to $\mathcal{P}$,
will be used as a perturbation term.

\begin{proposition}\label{generalcompu2}  
Let $p\in (2,3]$ and let $\Omega$ be any smooth bounded domain.

(i)  For any $\kappa>0$ sufficiently small,  there exist $M,\delta>0$ (depending on $u$ and $\kappa$), such that
\begin{equation}\label{auxfctB2}
\mathcal{P}\Bigl[u^{p-1}d_{\Omega}^{2-p+\kappa}\Bigr]\le -c(p)\kappa u^{p-1}d_{\Omega}^{-p+\kappa} 
   \quad\textrm{in} \; \Omega_\delta\times [T/2, T)
    \end{equation}
where $c(p)>0$ and
    \begin{equation}\label{auxfctB1}
\mathcal{P}\Bigl[u^{p-1-\kappa}d_{\Omega}^{2-p+\kappa}\Bigr]\le M u^{p-1-\kappa}d_{\Omega}^{-p+\kappa+ 1} 
  \quad\textrm{in} \; \Omega_\delta\times [T/2, T).
   \end{equation}

 (ii)  Assume that $u$ satisfies
 \begin{equation}\label{hypsharpprofile2}
|\nabla u|\le d_\Omega^{-\beta}\bigl(d_p +Cd_\Omega^\alpha\bigr)\quad\hbox{ in $\Omega\times (0, T),$}
\end{equation}
for some $\alpha\in (0,\beta]$ and $C>0$.
Then there exist $M,\delta>0$ such that 
\begin{equation}\label{auxfctB3}
\mathcal{P}\Bigl[u^{p-1}d_{\Omega}^{2-p}\Bigr]\le 
 Mu^{p-1}d_\Omega^{-p+\alpha} 
  \quad\textrm{in} \; \Omega_\delta\times [T/2, T).
   \end{equation}
   \end{proposition}

 We note that, in view of Theorems~\ref{convgradest} and \ref{shgradest},
assumption \eqref{hypsharpprofile2} is in particular satisfied with $\alpha\in(0,\frac{1}{2(p-1)})$ if $\Omega$ is a ball or an annulus,
or with $\alpha\in(0,\frac{3-p}{4(p-1)})$ if $\Omega$ is convex and $p<3$.

\begin{proof} We shall apply Lemma~\ref{generalcompu} with $\kappa\in [0,p-2)$ and $\sigma\in \{0,\kappa\}$.
We will eventually particularize the different choices: $\sigma=\kappa=0$ (for assertion (ii)),
$\sigma=\kappa>0$ (for \eqref{auxfctB1}), $\sigma=0$, $\kappa>0$ (for \eqref{auxfctB2}).
Set
\begin{equation}\label{deftildeP}
\tilde{\mathcal P}:= \dfrac{-\mathcal P\bigl[u^{p-1-\sigma}d_\Omega^{2+\kappa-p}\bigr]} {u^{p-1-\sigma}d_\Omega^{\kappa-p}}.
   \end{equation}
By Lemma~\ref{generalcompu}, we have
\begin{equation}\label{auxfctRecallLemma}
\tilde{\mathcal P}\ge F(X)+d_\Omega G^{p-1}\bigl[(p-1-\sigma)(p-1)X-p(p-2-\kappa)\bigr]-Md_\Omega
   \end{equation}
in $\Omega_{\delta_0}\times[T/2,T)$, with 
$$F(X)=(p-1-\sigma)(p-2-\sigma)X^2-2(p-1-\sigma)(p-2-\kappa)X+(p-1-\kappa)(p-2-\kappa).$$
By an elementary calculation, we note that
\begin{equation}\label{dGpelem1}
F(X)=(p-1-\kappa)(p-2-\kappa)[X-1]^2\quad\hbox{ if $\sigma=\kappa$},
\end{equation}
 and that 

\begin{equation}\label{dGpelem2}
\inf_{X\in\R} F(X)=\dfrac{\kappa(p-2-\kappa)}{p-2} \quad\hbox{ if $\sigma=0$.}
\end{equation}
Also, by \eqref{FiPucciSouplet}
 or \eqref{hypsharpprofile2}, we have
$$
|\nabla u|\leq d_{\Omega}^{-\beta}\bigl[(1+\eta) d_p+C(\eta)d_{\Omega}^\alpha\bigr]\quad\text{in}\ \Omega\times (0, T),
$$
with
\begin{equation}\label{dGp3}
\begin{cases}
&\alpha= \beta \hbox{and any $\eta>0$; or} \\ 
&\alpha \ \hbox{taken from (\ref{hypsharpprofile2}) and $\eta=0$, assuming (\ref{hypsharpprofile2}),}
\end{cases}
\end{equation}
and we will assume these conditions on $\alpha, \eta$ in what follows.

Integrating in the normal direction, we get
$$u\leq d_{\Omega}^{1-\beta}\bigl[(1+\eta) c_p+C(\eta)d_{\Omega}^\alpha\bigr]\quad\text{in}\quad\Omega\times (0, T).$$
Consequently,
\begin{equation}\label{dGp1}
d_{\Omega}G^{p-1}\le \dfrac{1+\eta}{p-1}+C(\eta)d_\Omega^\alpha \quad\text{in}\;\Omega\times [0, T)
\end{equation}
and
\begin{equation}\label{dGp2}
d_\Omega^{2-p}u^{p-1}\le c_p^{p-1}(1+\eta)+ C(\eta)d_{\Omega}^\alpha\quad\text{in}\;\Omega\times [0, T),
\end{equation}

We split the discussion into three subregions of the cylinder $Q:=\Omega_{\delta_0}\times [T/2, T)$ relative to the variable $X$.

\noindent{\bf CASE 1.} $\Sigma_1:=Q\cap\bigl\{X\ge\frac{p(p-2-\kappa)}{(p-1) (p-1-\sigma)}\bigr\}$. By \eqref{dGpelem1}-\eqref{dGpelem2}, 
we immediately have
\begin{equation}\label{LemauxfctB1}
\tilde{\mathcal P}\ge -Md_\Omega\quad\hbox{ in $\Sigma_1$ if $\sigma=\kappa$},
\end{equation}
and
\begin{equation}\label{LemauxfctB2}
\tilde{\mathcal P}\ge \dfrac{\kappa(p-2-\kappa)}{p-2} -Md_\Omega\quad\hbox{ in $\Sigma_1$ if $\sigma=0$}.
\end{equation}

\noindent{\bf CASE 2.}  $\Sigma_2:=Q\cap\bigl\{\frac{p-2}{p-1}\le X\le\frac{p(p-2-\kappa)}{(p-1) (p-1-\sigma)}\bigr\}$. 

By \eqref{dGp1}, for $\alpha,\eta$ as in \eqref{dGp3}, we have
\begin{align*}
\tilde{\mathcal P}
&\ge \Bigl({1+\eta\over p-1}+C(\eta)d_\Omega^\alpha\Bigr)\bigl[(p-1-\sigma)(p-1)X-p(p-2-\kappa)\bigr] + F(X)-Md_\Omega\\
&\ge {1+\eta\over p-1}\bigl[(p-1-\sigma)(p-1)X-p(p-2-\kappa)\bigr] + F(X)-M_1d_\Omega^\alpha,
\end{align*}
 hence
\begin{equation}\label{tildeP2}
 \tilde{\mathcal P}\ge\, (p-1-\sigma)(p-2-\sigma)G(X)-M_1d_\Omega^\alpha
\end{equation}
in $\Sigma_2$, where 
\begin{align*}
G(X)&=X^2-2{p-2-\kappa\over p-2-\sigma}X+{(p-1-\kappa)(p-2-\kappa)\over (p-1-\sigma)(p-2-\sigma)} \\
&\quad +{1+\eta\over p-1}\Bigl[{p-1\over p-2-\sigma}X-{p\over p-1-\sigma}{p-2-\kappa\over p-2-\sigma}\Bigr].
\end{align*}
 We now treat separately the three subcases relative to $\kappa,\sigma$.

{\bf 2.1.} Let us first consider the case $\kappa=\sigma=0$,
 under assumption (\ref{hypsharpprofile2})  (so that we can take $\eta=0$).
We have 
$$G\textstyle\bigl(\frac{p-2}{p-1}\bigr)={1\over (p-1)^2}+{1\over p-1}\bigl(1-{p\over p-1})=0,
\qquad
G'\textstyle\bigl(\frac{p-2}{p-1}\bigr)={-2\over p-1}+{1\over p-2}={3-p\over (p-1)(p-2)}\ge 0,$$
due to $p\le 3$.\footnote{This is where the restriction $p\le 3$ in our results crucially enters.
If $p>3$, it does not seem possible to construct an auxiliary function for which the parabolic operator ${\mathcal P}$
fulfills the required sign properties in order to apply the maximum principle and obtain the correct GBU rate $1/(p-2)$.}
 Since $G$ is convex, it follows that $G(X)\ge 0$ for all $X\ge\frac{p-2}{p-1}$, hence
\begin{equation}\label{LemauxfctB3}
\tilde{\mathcal P}\ge -M_1d_\Omega^\alpha\quad\hbox{ in $\Sigma_2$ for $\kappa=\sigma=0$, with $\eta=0$.}
\end{equation}

{\bf 2.2.} Next consider the case $0<\sigma=\kappa<p-2$. We have
$$
G(X)=(X-1)^2+{1+\eta\over p-1}\Bigl[{p-1\over p-2-\kappa}X-{p\over p-1-\kappa}\Bigr],
$$
hence
$$
G\textstyle\bigl(\frac{p-2}{p-1}\bigr)
={1\over (p-1)^2}+{1+\eta\over p-1}\bigl[{p-2\over p-2-\kappa}-{p\over p-1-\kappa}\bigr]\\
={1\over (p-1)^2}\bigl\{1-(1+\eta){(p-1)(p-2-2\kappa)\over (p-2-\kappa)(p-1-\kappa)}\bigr\}.
$$
Choosing $\eta={\kappa\over 2(p-1)(p-2)}$ and using
\begin{align*}
(1+\eta)\textstyle{(p-1)(p-2-2\kappa)\over (p-2-\kappa)(p-1-\kappa)}
&=\textstyle\bigl(1+{\kappa\over 2(p-1)(p-2)}\bigr)\bigl(1-{\kappa\over p-1}\bigr)^{-1}\bigl(1-{\kappa\over p-2}\bigr)^{-1}\bigl(1-{2\kappa\over p-2}\bigr) \\
&=1-\textstyle{\kappa\over 2(p-1)(p-2)}+O(\kappa^2), \quad\hbox{ as $\kappa\to 0^+$,}
\end{align*}
it follows that $G\bigl(\frac{p-2}{p-1}\bigr)\ge {\kappa\over 3(p-1)^3(p-2)}$ for $\kappa>0$ sufficiently small. On the other hand we have
$G'\bigl(\frac{p-2}{p-1}\bigr)\ge  -{2\over p-1}+{1\over p-2}= {3-p\over (p-1)(p-2)}\ge 0$, due to $p\le 3$.
 Since $G$ is convex, it follows that $G(X)\ge {\kappa\over 3(p-1)^3(p-2)}$ for all $X\ge\frac{p-2}{p-1}$ hence, by \eqref{tildeP2},
\begin{equation}\label{LemauxfctB4}
\tilde{\mathcal P}\ge {\kappa\over 6(p-1)^2}-M_1d_\Omega^\alpha\quad\hbox{ in $\Sigma_2$ with $\kappa=\sigma>0$} small.
\end{equation}

 {\bf 2.3.} Then consider the case $\sigma=0<\kappa<p-2$. We have
$$G(X)=X^2-\textstyle{2(p-2-\kappa)\over p-2}X+{(p-1-\kappa)(p-2-\kappa)\over (p-1)(p-2)} 
+{1+\eta\over p-1}\bigl[{p-1\over p-2}X-{p(p-2-\kappa)\over (p-1)(p-2)}\bigr],$$
hence
\begin{align*}
\textstyle G({p-2\over p-1})
&=\textstyle\bigl({p-2\over p-1}\bigr)^2-{2(p-2-\kappa)\over p-1}+{(p-1-\kappa)(p-2-\kappa)\over (p-1)(p-2)}  
+{1+\eta\over p-1}\bigl[1-{p(p-2-\kappa)\over (p-1)(p-2)}\bigr] \\
&\ge\textstyle\bigl({p-2\over p-1}\bigr)^2-{2(p-2)\over p-1}+{2\kappa\over p-1}+1-{(2p-3)\kappa\over (p-1)(p-2)}-{1\over (p-1)^2}-{\eta\over (p-1)^2} 
+{p\kappa\over (p-1)^2(p-2)}\\
& =\textstyle{(p-2)^2-2(p-2)(p-1)+ (p-1)^2-1\over (p-1)^2}+{2(p-1)(p-2)-(2p-3)(p-1)+p\over (p-1)^2(p-2)} \kappa-{\eta\over (p-1)^2} \\
&=\textstyle
{\kappa\over (p-1)^2(p-2)}-{\eta\over (p-1)^2}\ge {\kappa\over 2(p-1)^2(p-2)}
\end{align*}
upon choosing $\eta={\kappa\over 2(p-2)}$,
and
$$
\textstyle G'({p-2\over p-1})
=\textstyle {2(p-2)\over p-1}-{2(p-2-\kappa)\over p-2}+{1+\eta\over p-2}
\ge\textstyle {2(p-2)\over p-1}-{2(p-2)\over p-2}+{1\over p-2}={3-p\over (p-1)(p-2)}\ge 0.
$$
due to $p\le 3$. It follows from \eqref{tildeP2} that
\begin{equation}\label{LemauxfctB4b}
\tilde{\mathcal P}\ge {\kappa\over 2(p-1)}-M_1d_\Omega^\alpha\quad\hbox{ in $\Sigma_2$ with $\kappa>0$ small and $\sigma=0$.} 
\end{equation}

\noindent{\bf CASE 3.}  $\Sigma_3:=Q\cap\bigl\{X\le \frac{p-2}{p-1}\bigr\}$. Rewrite \eqref{auxfctRecallLemma} as
\begin{equation}\label{LemauxfctCase3}
\tilde{\mathcal P}\ge F(X)+(p-1-\sigma)(p-1){d_\Omega^2G^p\over u} -p(p-2-\kappa) d_\Omega G^{p-1}-Md_\Omega.
\end{equation}
Using Young's inequality, for any $\theta>0$, we estimate the third term of the RHS of \eqref{LemauxfctCase3} by
$$d_\Omega G^{p-1}=\Bigl({\theta d_\Omega^{2(p-1)/p}G^{p-1}\over u^{(p-1)/p}}\Bigr) \Bigl({d_\Omega^{(2-p)/p}u^{(p-1)/p}\over \theta }\Bigr)
\le {p-1\over p}\theta^{p/(p-1)} {d_\Omega^2G^p\over u} +{1\over p}{d_\Omega^{2-p}u^{p-1}\over \theta^p},$$
hence
$$p(p-2-\kappa) d_\Omega G^{p-1}\le (p-1)(p-2-\kappa) \theta^{p/(p-1)} {d_\Omega^2G^p\over u} +(p-2-\kappa){d_\Omega^{2-p}u^{p-1}\over \theta^p}.$$
Choose $\theta=\bigl({p-1-\sigma\over p-2-\kappa}\bigr)^{(p-1)/p}$. 
By \eqref{dGp2} and \eqref{defcp}, for $\alpha,\eta$ as in \eqref{dGp3}, we have
\begin{align*}
(p-1-\sigma)(p-1) &\dfrac{d_\Omega^2 G^p}{u}-p(p-2-\kappa)d_\Omega G^{p-1}\\
&\ge -(p-2-\kappa)\Bigl({p-2-\kappa\over p-1-\sigma}\Bigr)^{p-1}d_\Omega^{2-p}u^{p-1}\\
&\ge  -(p-2-\kappa)\Bigl({p-2-\kappa\over p-1-\sigma}\Bigr)^{p-1}(c_p^{p-1}(1+\eta)+  Cd_\Omega^\alpha)\\ 
&\ge  -(1+\eta){p-2-\kappa\over p-1}\Bigl[{(p-2-\kappa)(p-1)\over (p-1-\sigma)(p-2)}\Bigr]^{p-1}- Cd_\Omega^\alpha. 
\end{align*}
 Therefore, 
$$\tilde{\mathcal P}\ge \tilde{\mathcal P}_0-Md^\alpha_\Omega,\ \hbox{ with } 
\tilde{\mathcal P}_0=F(X)-(1+\eta){p-2-\kappa\over p-1}\Bigl[{(p-2-\kappa)(p-1)\over (p-1-\sigma)(p-2)}\Bigr]^{p-1}.$$
 We again treat separately the three subcases relative to $\kappa,\sigma$.

{\bf 3.1.} Let us first consider the case $\kappa=\sigma=0$,
 under assumption (\ref{hypsharpprofile2})  (so that we can take $\eta=0$).
In view of \eqref{dGpelem1} we have, for $0\le X\le \frac{p-2}{p-1}$,
$$\tilde{\mathcal P}_0\ge  (p-1)(p-2)\Bigl[1-\frac{p-2}{p-1}\Bigr]^2-{p-2\over p-1} =0,$$
hence
\begin{equation}\label{LemauxfctB5}
\tilde{\mathcal P}\ge  -M_1d_\Omega^\alpha\quad\hbox{ in $\Sigma_3$  for $\kappa=\sigma=0$, with $\eta=0$.}
\end{equation}

{\bf 3.2.} Next consider the case $0<\kappa=\sigma<p-2$.
 Using again \eqref{dGpelem1} we have, for $0\le X\le \frac{p-2}{p-1}$, 
\begin{align*}
\tilde{\mathcal P}_0
&\ge {(p-1-\kappa)(p-2-\kappa)\over (p-1)^2}
-(1+\eta){p-2-\kappa\over p-1}\Bigl[{(p-2-\kappa)(p-1)\over (p-1-\kappa)(p-2)}\Bigr]^{p-1}\\
&= {(p-1-\kappa)(p-2-\kappa)\over (p-1)^2}
\Bigl[1-(1+\eta)\Bigl({p-1\over p-1-\kappa}\Bigr)^p\Bigl({p-2-\kappa\over p-2}\Bigr)^{p-1}\Bigr] =: c_0(\kappa).
\end{align*}
Choosing $\eta={\kappa\over 2(p-1)(p-2)}$ and using
\begin{align*}
(1+\eta)\bigl(\textstyle{p-1\over p-1-\kappa}\bigr)^p\bigl({p-2-\kappa\over p-2}\bigr)^{p-1}
&=\bigl(1+\textstyle{\kappa\over 2(p-1)(p-2)}\bigr)\bigl(1-{\kappa\over p-1}\bigr)^{-p}\bigl(1-{\kappa\over p-2}\bigr)^{p-1}\\
&=1-\textstyle{\kappa\over 2(p-1)(p-2)} +O(\kappa^2),
\quad\hbox{ as $\kappa\to 0^+$,}
\end{align*}
we obtain
$$c_0(\kappa)
\ge  {(p-1-\kappa)(p-2-\kappa)\over (p-1)^2}{\kappa\over 3(p-1)(p-2)} 
\ge  {\kappa\over 6(p-1)^2},\quad\hbox{ as $\kappa\to 0^+$,}
$$
hence
\begin{equation}\label{LemauxfctB6}
\tilde{\mathcal P}\ge {\kappa\over 6(p-1)^2}-M_1d_\Omega^\alpha\quad\hbox{ in $\Sigma_3$ with $\kappa=\sigma>0$ small.}
\end{equation}

 {\bf 3.3.} Then consider the case $0<\kappa<p-2$ and $\sigma=0$.
Since now
$$F(X)=(p-1)(p-2)X^2-2(p-1)(p-2-\kappa)X+(p-1-\kappa)(p-2-\kappa),$$
we have
$$F'(\textstyle\frac{p-2}{p-1})=2(p-2)^2-2(p-1)(p-2-\kappa)=2(2-p+(p-1)\kappa)<0$$
for $\kappa>0$ small. Since $F$ is convex, it follows that, for $0\le X\le \frac{p-2}{p-1}$,
\begin{align*}
F(X)\ge F(\textstyle\frac{p-2}{p-1})
&=\textstyle{(p-2)^3\over p-1}-2(p-2)(p-2-\kappa)+(p-1-\kappa)(p-2-\kappa)\\
&=\textstyle{(p-2)^3\over p-1}+(p-2-\kappa)(3-p-\kappa)
\ge {p-2\over p-1}-\kappa,
\end{align*}
hence
\begin{align*}
\tilde{\mathcal P}_0
&\ge \textstyle{p-2\over p-1}-\kappa
-(1+\eta){p-2-\kappa\over p-1}\bigl({p-2-\kappa\over p-2}\bigr)^{p-1} 
= \textstyle{p-2\over p-1}-\kappa
-(1+\eta){p-2\over p-1}\bigl(1-{\kappa\over p-2}\bigr)^p\\
& \ge \textstyle{p-2\over p-1}-\kappa
-(1+\eta)\bigl({p-2\over p-1}- {p\over p-1}\kappa+O(\kappa^2)\bigr)\ge \textstyle{1\over p-1}\kappa- \textstyle{p-2\over p-1}\eta+O(\kappa^2),
\end{align*}
as $\kappa\to 0^+$. Choosing $\eta=\textstyle{\kappa\over 2(p-2)}$ we obtain 
\begin{equation}\label{LemauxfctB6b}
\tilde{\mathcal P}\ge {\kappa\over 2(p-1)}-M_1d_\Omega^\alpha\quad\hbox{ in $\Sigma_3$ for $\kappa>0$ small and $\sigma=0$.}
\end{equation}

Collecting formulae \eqref{LemauxfctB1}-\eqref{LemauxfctB4b} and \eqref{LemauxfctB5}-\eqref{LemauxfctB6b}, we conclude that:

\noindent  $\bullet$ For $\sigma=\kappa=0$,  assuming (\ref{hypsharpprofile2}), and $\delta>0$ sufficiently small,
 $$ \tilde{\mathcal P}\ge  -M_1d_\Omega^\alpha\quad\hbox{ in $\Omega_\delta\times [T/2, T)$};$$
\par\noindent $\bullet$ For $\sigma=\kappa>0$ sufficiently small and $\delta=\delta(\kappa)>0$ sufficiently small,
 $$\tilde{\mathcal P}\ge  -M_1 d_\Omega\quad\hbox{ in $\Omega_\delta\times [T/2, T)$};$$ 
\par\noindent $\bullet$ For $\sigma=0$, $\kappa>0$ sufficiently small and $\delta=\delta(\kappa)>0$ sufficiently small,
 $$\tilde{\mathcal P}\ge  c(p)\kappa\quad\hbox{ in $\Omega_\delta\times [T/2, T)$},$$
with $c(p)>0$. 
 In view of \eqref{deftildeP},  this yields the conclusions of the proposition.
\end{proof}

\subsection{Proof of Theorems \ref{Thmain1}-\ref{Thmain1b}  and of Proposition~\ref{Thmain1gen}}

 We only need to establish Theorem \ref{Thmain1b} and Proposition~\ref{Thmain1gen}.
Indeed, Theorem \ref{Thmain1} is a consequence of Proposition~\ref{Thmain1gen} since,
in view of Theorems~\ref{convgradest} and \ref{shgradest},
assumption \eqref{hypsharpprofile} of Proposition~\ref{Thmain1gen} is satisfied for $\alpha\in(0,\frac{1}{2(p-1)})$ if $\Omega$ is a ball or an annulus,
or for $\alpha\in(0,\frac{3-p}{4(p-1)})$ if $\Omega$ is convex and $p<3$.

We define the following auxiliary function
$$J=u_t-\eps H_1,\qquad H_1:=u^{p-1-\kappa}d_{\Omega}^{2-p+\kappa}\bigr[1+u^\kappa\bigl]$$
for general domains, and
$$J=u_t-\eps H_2,\qquad H_2:= u^{p-1}d_{\Omega}^{2-p}\bigr[1+d_\Omega^\kappa\bigl]$$
for convex domains or annuli,
with $\eps>0$ to be determined.

 For any $\kappa\in (0,p-2)$ sufficiently small, we deduce from Proposition~\ref{generalcompu2} 
that there exist $\delta_\kappa,M_\kappa,M>0$ such that 
\begin{align*}
\mathcal{P}H_1
&=\mathcal{P}\bigl(u^{p-1-\kappa}d_{\Omega}^{2-p+\kappa}\bigr)+\mathcal{P}\bigl(u^{p-1}d_{\Omega}^{2-p+\kappa}\bigr)\\
&\le  M_\kappa u^{p-1-\kappa}d_{\Omega}^{-p+\kappa+1}-c(p)\kappa u^{p-1}d_{\Omega}^{-p+\kappa} 
\le   M_\kappa c_0^{-\kappa}u^{p-1}d_{\Omega}^{-p+1}-c(p)\kappa u^{p-1}d_{\Omega}^{-p+\kappa}  \\
&\le u^{p-1}d_{\Omega}^{-p+\kappa}\bigl[ M_\kappa c_0^{-\kappa}d_{\Omega}^{1-\kappa}-c(p)  \kappa\bigl]<0
\quad\hbox{ in $\Omega_{\delta_\kappa}\times [T/2, T)$,}
\end{align*}
where we also used \eqref{Hopfucd}, and
\begin{align*}
\mathcal{P}H_2
&=\mathcal{P}\bigl(u^{p-1}d_{\Omega}^{2-p}\bigr)+\mathcal{P}\bigl(u^{p-1}d_{\Omega}^{2-p+\kappa}\bigr)
\le M u^{p-1}d_{\Omega}^{-p+\alpha}-c(p)\kappa u^{p-1}d_{\Omega}^{-p+\kappa}  \\
&\le u^{p-1}d_{\Omega}^{-p+\kappa}\bigl[M d_{\Omega}^{\alpha-\kappa}-c(p)  \kappa\bigl]<0
\quad\hbox{ in $\Omega_{\delta_\kappa}\times [T/2, T)$.}
\end{align*}
Since $\mathcal{P}(u_t)=0$, it follows that 
\begin{equation}\label{PMJ0}
\mathcal{P}J>0\quad\hbox{ in $\Omega_{\delta_\kappa}\times [T/2, T)$.}
\end{equation}

We next examine the initial-boundary conditions for $J$. Recall that  $u_t=0$  on $\partial\Omega\times (0, T)$.
For each $t\in (0, T)$, we have
$$  
u^{p-1}d_\Omega^{2-p}\le \|\nabla u(t)\|_\infty^{p-1} d_\Omega,\quad
u^{p-1-\kappa}d_\Omega^{2+\kappa-p}\le \|\nabla u(t)\|_\infty^{p-1-\kappa} d_\Omega,
$$ 
so that, after extension by continuity, we have
\begin{equation}\label{PMJ2}
J=0 \quad\hbox{ on $\partial\Omega\times (0, T)$.}
\end{equation}
From  assumption \eqref{hypmonot} (or, more generally, \eqref{hypmonot3b}) and since,
as a consequence of \eqref{Bernstein0} and parabolic estimates, 
$u$ extends to a classical solution in $\Omega\times (0,T]$,
it follows from the strong maximum principle that $u_t>0$ in $\Omega_\eta\times (t_0,T]$,
 with $t_0:=\min(T/2,T-\eta)$.
Therefore, for any $\delta>0$ small, there exists $\bar c(\delta)>0$ such that 
\begin{equation}\label{PMJ3}
u_t\geq \bar c(\delta) \quad\hbox{ on $\left\{x\in \Omega, d_\Omega(x)= \delta\right\}\times (t_0, T)$}
\end{equation}
and, by Hopf's Lemma, there exist $\hat c, \hat\delta>0$ such that
\begin{equation}\label{PMJ4}
u_t(\cdot,t_0)\ge \hat c d_\Omega \quad\hbox{ in $\Omega_{\hat\delta}$.}
\end{equation}
Also, it follows from \eqref{FiPucciSouplet} and \eqref{Hopfucd} that there exists a constant $C_1>0$ (independent of $\kappa$ small)
such that
\begin{equation}\label{PMJ5}
H_1, H_2\le C_1 \quad\hbox{ in $\Omega\times [T/2, T)$.}
\end{equation}
 Choosing $\eps>0$ sufficiently small (depending on $\kappa$), it follows from \eqref{PMJ2}-\eqref{PMJ5} that
\begin{equation}\label{PMJ6}
J\geq 0\quad\textrm{on}\; \partial_p(\Omega_{\delta_\kappa}\times [t_0, T)),
\end{equation}
where $\partial_p$ denotes the parabolic boundary.

Now, in view of \eqref{PMJ0}, \eqref{PMJ6} we may apply the maximum principle to deduce that 
$J\geq 0$ in $\Omega_{\delta_\kappa}\times (t_0, T)$. Dividing this inequality by $d_\Omega$, we get that 
\begin{equation}\label{utuqd}
\dfrac{u_t}{d_\Omega}\geq \eps\Bigl(\dfrac{u}{d_\Omega}\Bigr)^q \quad\hbox{ in $\Omega_{\delta_\kappa}\times (t_0, T)$},
\end{equation}
where
$$
q=\begin{cases}
\ p-1-\kappa, &\hbox{  for any small $\kappa>0$, in the case of Theorem \ref{Thmain1b}  
 (with $\eps=\eps(\kappa)$), }\\
\ p-1, &\hbox{ in the case of Proposition~\ref{Thmain1gen}}
\end{cases}
$$
(and in the latter case we have used for $H_2$ a single value of $\kappa>0$ small).
It follows that
$$\partial_t u_\nu\geq \eps (u_\nu)^q \quad\hbox{ on $\partial\Omega\times (t_0, T)$}.$$
Indeed, for any $\bar x\in\partial\Omega$ it suffices to apply \eqref{utuqd} with $x=\bar x+s \nu(\bar x)$, so that $d_\Omega(x)=s$,
and to use $\underset{s\to 0}{\lim} u_t(\bar x+s \nu(\bar x),t)/s=\partial_tu_\nu(\bar x, t)$.
By integration, it follows that
$$u_\nu(x,t)\leq \bigl[(q-1)\eps(T-t)\bigr]^{-\frac{1}{q-1}} \quad\hbox{ on $\partial\Omega\times (t_0, T)$}.$$
Finally, we know (see \cite[Proposition~2.3]{SZ})
 that, as a consequence of the maximum principle,
$$\sup_{\overline\Omega\times [t_0,t]} |\nabla u|\le \max\Bigl\{\|\nabla u(t_0)\|_\infty, \sup_{\partial\Omega\times [t_0,t]} |u_\nu|\Bigr\}
+T\|\nabla h\|_\infty, \quad t_0<t<T,$$
hence
$$\|\nabla u(t)\|_\infty\le \|\nabla u(t_0)\|_\infty+T\|\nabla h\|_\infty+\bigl[(q-1)\eps(T-t)\bigr]^{-\frac{1}{q-1}},\quad t_0<t<T,$$
and the conclusion follows.

\begin{remark}\label{remaltproof}
 Theorem \ref{Thmain1b} can alternatively be proved by using an auxilary function of the form
\begin{equation}\label{alternativeJ}
J(x,t):= u_t-\eps \bigr[u|\nabla u|^{p-2}+u^{p-1}d_{\Omega}^{2-p+\kappa}\bigl]
\end{equation}
(the rest of the proof being otherwise similar).
However this alternative method does not seem to provide the sharp rate in Theorem~\ref{Thmain1}.
Also, instead of the gradient term in \eqref{alternativeJ}, one can more generally consider a term of the form 
$$d_\Omega^{1-a}u^a|\nabla u|^{p-1-a}$$ 
with a parameter $a\in\R$. But it does not seem possible to remove the restriction $p\le 3$ in this way.
\end{remark}

\section{More singular rate for minimal GBU solutions: proof of Theorem~\ref{thmfastratemain}}\label{sect5}

Theorem~\ref{thmfastratemain} will be proved through a series of lemmas (the first two are valid in any space dimension).
The first auxiliary lemma gives a gradient estimate for a linear heat equation with drift.
The argument of proof is the same as in \cite{QSbook07, PS3}, but we here need the estimate in a more quantitative form. 

\begin{lemma}
\label{gradbound1}
Let $t_0<t_1$, let the vector field $B$ be H\"older continuous on 
$\overline\Omega\times[t_0,t_1]$ and let $w\in C(\overline\Omega\times[t_0,t_1])\cap C^{2,1}(\overline\Omega\times(t_0,t_1])$
be a classical solution of 
\be
\left\{\begin{array}{llll}
\hfill w_t-\Delta w&=&B\cdot \nabla w,& x\in\Omega,\ t_0< t\le t_1,
\vspace{1mm} \\
\hfill w&=&0,& x\in\partial\Omega,\ t_0\le t\le t_1,
\end{array}
\right.
\ee
such that
$$|w|\le M_1,\quad |B|\le M_2 \quad\hbox{ in ${\Omega\times[t_0,t_1]}$. }$$
Then 
\be \label{estimgradw}
\|\nabla w(\cdot,t_1)\|_\infty\le C(\Omega)M_1\max\bigl((t_1-t_0)^{-1/2},M_2\bigr).
\ee
\end{lemma}

\begin{proof}[Proof of Lemma~\ref{gradbound1}]
Let $s\in [t_0,t_1)$, and put
$$K=\sup_{\sigma\in [0,t_1-s]} \sigma^{1/2}\|\nabla w(s+\sigma)\|_\infty.$$
For $\tau\in (0,t_1-s)$, by the variation-of-constants formula, we have
$$w(s+\tau)=e^{\tau\Delta}w(s)+
\int_0^\tau e^{(\tau-\sigma)\Delta}(B\cdot\nabla w)(s+\sigma)\, d\sigma,$$
where $(e^{t\Delta})_{t\ge 0}$ denotes the Dirichlet heat semigroup on $\Omega$.
Using standard smoothing properties of $e^{t\Delta}$
and the fact that
$$\int_0^\tau (\tau-\sigma)^{-1/2}\sigma^{-1/2}\, d\sigma=\int_0^1 (1-z)^{-1/2}z^{-1/2}\, dz,$$
it follows that
$$\begin{aligned}
\|\nabla w(s+\tau)\|_\infty
&\leq C_1\tau^{-1/2}\|w(s)\|_\infty+
C_1 \int_0^\tau (\tau-\sigma)^{-1/2}\|B\cdot\nabla w(s+\sigma)\|_\infty\, d\sigma\\
&\leq C_1M_1\tau^{-1/2}+C_2M_2K.
\end{aligned}$$
where the constants $C_1, C_2>0$ depend only on $\Omega$.
Multiplying by $\tau^{1/2}$ and taking the supremum for $\tau\in [0,t_1-s]$, we obtain
$$K\leq C_1M_1+C_2(t_1-s)^{1/2}M_2K.$$
Now choosing $s=\max\bigl(t_0,t_1-(C_2M_2)^{-2})$, hence $C_2(t_1-s)^{1/2}M_2\le 1/2$, we obtain
$K\leq 2C_1M_1$. Therefore,
$$\|\nabla w(\cdot,t_1)\|_\infty\leq 2C_1M_1(t_1-s)^{-1/2}\leq 2C_1M_1\max((t_1-t_0)^{-1/2},C_2M_2),$$
which proves the Lemma. 
\end{proof}

Our next lemma states a similar conclusion for the normal derivative on the boundary, 
but assuming only an upper bound on the solution.

\begin{lemma}
\label{gradbound2}
Let $t_0<t_1$, let $B$ be H\"older continuous on $\overline\Omega\times[t_0,t_1]$ 
and let $v\in C(\overline\Omega\times[t_0,t_1])\cap C^{2,1}(\overline\Omega\times(t_0,t_1])$
be a classical solution of 
\be \label{eqnvgrad}
\left\{\begin{array}{llll}
\hfill v_t-\Delta v&=&B\cdot \nabla v,& x\in\Omega,\ t_0< t\le t_1,
\vspace{1mm} \\
\hfill v&=&0,& x\in\partial\Omega,\ t_0\le t\le t_1,
\end{array}
\right.
\ee
such that
$$v\le M_1 \quad\hbox{ in ${\Omega\times[t_0,t_1]}$}$$
and
$$|B|\le M_2 \quad\hbox{ in ${\Omega\times[t_0,t_1]}$. }$$
Then 
\be \label{estimvnu}
\frac{\partial v}{\partial \nu}(\cdot,t_1)\le C(\Omega)M_1\max\bigl((t_1-t_0)^{-1/2},M_2\bigr)
\quad\hbox{ on $\partial\Omega$. }
\ee
\end{lemma}

\begin{proof}[Proof of Lemma~\ref{gradbound2}]
Let $\phi(x)=\max(v(x, t_0),0)$. Since $\phi\in C_0(\Omega)$, the linear problem 
\be 
\left\{\begin{array}{llll}
\hfill w_t-\Delta w&=&B\cdot \nabla w,& x\in\Omega,\ t_0< t\le t_1,
\vspace{1mm} \\
\hfill w&=&0,& x\in\partial\Omega,\ t_0\le t\le t_1,
\vspace{1mm} \\
\hfill w(x,t_0)&=&\phi(x),& x\in\Omega,
\end{array}
\right.
\ee
admits a unique solution $w\in C([t_0,t_1]\times \overline\Omega)\cap C^{1,2}((t_0,t_1]\times \overline\Omega)$.
Moreover, we have $0\le w\le M_1$ and $v\le w$ in $(t_0,t_1]\times \overline\Omega$ by the maximum principle.
It follows from Lemma~\ref{gradbound1} that
$$\|\nabla w(\cdot,t_1)\|_\infty\le C(\Omega)M_1\max\bigl((t_1-t_0)^{-1/2},M_2\bigr).$$
Since $v=w=0$ on $\partial\Omega$, we deduce that
$$\frac{\partial v}{\partial \nu}(\cdot,t_1)\le \frac{\partial w}{\partial \nu}(\cdot,t_1)\le C(\Omega)M_1\max\bigl((t_1-t_0)^{-1/2},M_2\bigr)
\quad\hbox{ on $\partial\Omega$. }$$
\end{proof}

In \cite{GH, QSbook07, PS3}, the general 
 lower bound $\|\nabla u(t)\|_\infty\ge C(T-t)^{-1/(p-2)}$ was obtained by means of
gradient estimates similar to \eqref{estimgradw}, combined with the boundedness of~$u_t$.
As for the proof in \cite{PS3} that $\lim_{t\to T}(T-t)^{1/(p-2)}\|\nabla u(t)\|_\infty=\infty$ 
for minimal GBU solutions under the assumptions of Theorem~\ref{thmfastratemain},  
it made use of an additional continuity property of $u_t$ near the ``corner''~$(0,T_-)$.
Here, as an improvement on~\cite{PS3}, so as to get 
the more singular estimate \eqref{uxmoresingular}, we will use a more precise upper vanishing estimate of $u_t$ near $t=T$,
which turns out to be satisfied by minimal GBU solutions under the assumptions of Theorem~\ref{thmfastratemain}.

\begin{lemma} \label{boundforut}
Let $u_0$ be as in Theorem~\ref{thmfastratemain}. Fix $t_0\in (0,T)$.
If $u$ is a minimal GBU solution, then there exists a constant $M>0$ such that
\be\label{boundut}
u_t(x,t)\le M(T-t)\quad\hbox{ in $[0,1]\times[t_0,T)$.}
\ee
\end{lemma}

\begin{proof}[Proof of Lemma~\ref{boundforut}]
We first claim that there exists a constant $M>0$ such that
$$u_{tt}>-M\quad\hbox{ in $[0,1]\times[t_0,T)$.}$$ 
Indeed, the function $w:=u_{tt}$ solves  the equation  
\be\label{boundutt}
w_t - w_{xx}=p|u_{x}|^{p-2}u_xu_{ttx}+p(p-1)|u_{x}|^{p-2}(u_{xt})^2\ge p|u_{x}|^{p-2}u_xw_x\,.
\ee
Since $w=u_{tt}=0$ on the boundary the claim follows from the maximum principle.

On the other hand, under the assumptions of the lemma, 
it was shown in \cite[Proposition~7.1]{PS3},
as a consequence of zero-number properties of $u_t$ for minimal GBU solutions, that
$$
u_t(x,T)\le 0,\quad x\in (0,1).
$$
Therefore, for all $x\in (0,1)$, using \eqref{boundutt}, we obtain
$$u_t(x,t)=u_t(x,T)-\int_t^{T} u_{tt}(x,s)\,ds\le -\int_t^{T} u_{tt}(x,s)\,ds\le M(T-t),$$
which proves the Lemma.
\end{proof} 

We are now in a position to prove Theorem~\ref{thmfastratemain}.

\begin{proof}[Proof of Theorem~\ref{thmfastratemain}]
Set 
$$m(t):=\|u_x(\cdot,t)\|_\infty,\quad 0\le t<T.$$
By the maximum principle and the symmetry of $u$, for each $t\in (0,T)$, we have
$$\sup_{s\in [0,t]}m(s)\le \max\Bigl(m(0),\sup_{s\in(0,t]} u_x(0,s)\Bigr).$$
Also, by \cite[Lemma~6.4]{PS3}, we know that for all $t\in (0,T)$, there exists $x_0(t)\in (0,1/2)$ such that $u_t(x,t)>0$ on $(0,x_0(t))$,
hence $u_{tx}(0,t)\ge 0$. Therefore, since $\lim_{t\to T}m(t)=\infty$, there exist $\eta>0$ such that
\be\label{mmaxgrad}
m(t)=u_x(0,t)>0\quad\hbox{and}\quad m'(t)=u_{tx}(0,t)\ge 0,\quad T-\eta\le t<T.
\ee

The function $v:=u_t$ satisfies \eqref{eqnvgrad} with $B=p|u_x|^{p-2}u_x$.
By \eqref{mmaxgrad}, we have
$$|B|\le pm^{p-1}(t)\quad\hbox{ in $[T-\eta,t)\times [0,1]$, for each $t\in [T-\eta,T)$}.$$
Let $t\in (T-\eta/2,T)$. We first apply \eqref{estimvnu}
 with $t_0=T-\eta$, $t_1=t$, $M_1=M\eta$ and $M_2=pm^{p-1}(t)$.
 Using \eqref{mmaxgrad} and \eqref{boundut}, we thus obtain
$$
m'(t)=u_{tx}(0,t)\le CM\eta\max\bigl((t-T+\eta)^{-1/2},m^{p-1}(t)\bigr)
\le C\bigl(1+m(t)\bigr)^{p-1}.
$$
(Here and below, $C$ denotes various positive constants possibly depending on the solution.)
Integrating over $(t,T)$, we deduce that
$1+m(t)\ge C(T-t)^{-1/(p-2)}$ on $(T-\eta/2,T)$,
hence
\be\label{firstestimate}
m(t)\ge C(T-t)^{-1/(p-2)},\quad T-\eta_1<t<T,
\ee
for some $\eta_1\in (0,\eta/2)$.

Let again $t\in (T-\eta/2,T)$. 
Recalling  \eqref{boundut}, we next apply \eqref{estimvnu} with $t_0=2t-T$, $t_1=t$, $M_1=M(T-t_0)=2M(T-t)$ and $M_2=pm^{p-1}(t)$, to get
$$u_{tx}(0,t)\le CM(T-t)\max\bigl((t-t_0)^{-1/2},m^{p-1}(t)\bigr).$$
Since $(p-1)/(p-2)>1/2$ and $t-t_0=T-t$, using \eqref{firstestimate}, we see that, 
for all $t\in (T-\eta_2,T)$ with $\eta_2\in (0,\eta/2)$ small enough, we have
$m^{p-1}(t)\ge C(T-t)^{-(p-1)/(p-2)}\ge (t-t_0)^{-1/2}$
hence
$$m'(t)=u_{tx}(0,t)\le CM(T-t)m^{p-1}(t),\quad T-\eta_2<t<T.$$
The conclusion follows by integrating over $(t,T)$.
\end{proof}

\medskip
{\bf Acknowledgements}  A. Attouchi  is supported by the Academy of Finland, project number
307870. Ph. Souplet  is partially supported by the Labex MME-DII (ANR11-LBX-0023-01).

\end{document}